\pdfoutput=1
\documentclass[psamsfonts,reqno]{article}

\usepackage{article}

\begin{document}
\author{Paige Randall North\thanks{This material is based upon work supported by the Air Force Office of Scientific Research under award number FA9550-16-1-0212.}}
\title{Type-theoretic weak factorization systems}
\date{\today}

\maketitle

\begin{abstract}
This article presents three characterizations of the weak factorization systems on finitely complete categories that interpret intensional dependent type theory with $\Sigma$-, $\Pi$-, and $\Id$-types. The first characterization is that the weak factorization system $(\mathsf L, \mathsf R)$ has the properties that $\mathsf L$ is stable under pullback along $\mathsf R$ and that all maps to a terminal object are in $\mathsf R$. We call such weak factorization systems \emph{type-theoretic}. The second is that the weak factorization system has an \emph{$\Id$-presentation}: roughly, it is generated by $\Id$-types in the empty context. The third is that the weak factorization system $(\mathsf L, \mathsf R)$ is generated by a \emph{Moore relation system}, a generalization of the notion of Moore paths.
\end{abstract}

\tableofcontents

\section{Introduction}

This paper is the second in a series (based upon the author's thesis \cite{Nor17}) in which we study categorical interpretations of dependent type theory of a certain species: display map categories. It has long been known that categorical interpretations of dependent type theory induce weak factorization systems \cite{GG08}, and so our goal is to characterize the weak factorization systems which harbor such interpretations.

In the first paper \cite{Nor19}, we considered display map categories modeling $\Sigma$- and $\Id$-types whose underlying category is Cauchy complete. We showed there that the induced weak factorization system is \emph{itself} a display map category modeling $\Sigma$- and $\Id$-types. This simplifies our problem: if we want to decide whether a given weak factorization system is induced by such a display map category modeling $\Sigma$- and $\Id$-types, then we only need to decide whether the weak factorization system itself is a display map category modeling $\Sigma$- and $\Id$-types. 

In the present paper, we turn to this problem: deciding whether a weak factorization system is a display map category modeling $\Sigma$- and $\Id$-types. Our main theorem is the following characterization:
\begin{mainthm}
Consider a category $\C$ with finite limits.
The following properties of any weak factorization system $(\mathsf L, \mathsf R)$ on $\C$ are equivalent:
\begin{enumerate}
\item it has an $\Id$-presentation;
\item it is type-theoretic;
\item it is generated by a Moore relation system;
\item $(\C, \mathsf R)$ is a display map category modeling $\Sigma$- and $\Id$-types.
\end{enumerate}
\end{mainthm}
\noindent A weak factorization system $(\mathsf L, \mathsf R)$ is \emph{type-theoretic} when all morphisms to the terminal object are in $\mathsf R$ (a necessary condition for $(\C, \mathsf R)$ to be a display map category) and when $\mathsf L$ is stable under pullback along $\mathsf R$ (a necessary condition for $(\C, \mathsf R)$ to model $\Pi$-types). Thus, the equivalence between (2) and (4) tells us that a category with a weak factorization system is a display map category modeling $\Sigma$- and $\Id$-types just when these two conditions hold.

To prove this equivalence, we introduce the notions of \emph{$\Id$-presentations} and, perhaps more interestingly, of \emph{Moore relation systems}. Roughly, a weak factorization has an {$\Id$-presentation} if it is induced by a model of $\Id$-types. A Moore relation system is an explicit algebraic presentation of the weak factorization systems under consideration. These are closely related to the path object categories of \cite{BG12} which are used to model identity types. However, a Moore relation system is a weaker notion than that of path object category, and what we lose in strictness, we make up for in the equivalence above.

This paper is organized as follows. Throughout, we fix a finitely complete category $\C$. In Section \ref{sec:relandfact}, we introduce the objects of study and the categories that contain them. In particular, we describe the categories $\fact[\C]$ of factorizations and $\rel[\C]$ of relations, their full subcategories $\ttWFS[\C] \hookrightarrow \fact[\C]$ of type-theoretic weak factorization systems and $\IdPres[\C] \hookrightarrow \rel[\C]$ of $\Id$-presentations, and functors $\reltofact: \rel[\C] \leftrightarrows \fact[\C]: \facttorel $. We then show in the following sections that these functors restrict to an equivalence $|\reltofact|: |\IdPres[\C] | \simeq |\ttWFS[\C]|: |\facttorel| $ where $| \text{ - } | $ denotes the proset truncation. We use the proset truncation because it makes isomorphism in $|\ttWFS[\C]|$ the usual notion of sameness between weak factorization systems (that is, having the same left and right classes of morphisms). In Section \ref{sec:mrs}, we describe Moore relation systems and show for that every Moore relation system $R$, the factorization $\reltofact (R)$ that it produces is a type-theoretic weak factorization system. In Section \ref{sec:bigdiagram}, we show that $\facttorel$ restricts to a functor $\ttWFS[\C] \to \IdPres[\C]$ and that $\reltofact \facttorel$ restricts to a functor $\ttWFS[\C] \to \ttWFS[\C]$ which is isomorpic to the identity functor under the proset truncation. In Section \ref{sec:idpresandmrs}, we show that a relation is a Moore relation system if and only if it is an $\Id$-presentation and that $\facttorel \reltofact : \IdPres[\C] \to \IdPres[\C]$ is isomorphic to the identity functor under proset truncation. 
We show in Section \ref{sec:idtypes} that having an $\Id$-presentation is equivalent to modeling $\Id$-types.
In Section \ref{sec:summary}, we put these results together to obtain Theorem \ref{mainthm}.

\section{Preliminaries}
\label{sec:relandfact}

This section is devoted to developing the concepts that we will study in the following sections. Our main theorem, \ref{mainthm}, is a comparison of a certain kind of relation (models of $\Id$-types) and a certain kind of factorization (weak factorization systems). We start with a categorical analysis of such relations and factorizations, and then we define the particular instances in which we are interested.

As mentioned above, we fix a finitely complete category $\C$ throughout this paper.
\subsection{Relations, relational factorizations, and factorizations}
 In this subsection, we define the fundamental objects: relations on $\C$, factorizations on $\C$, relational factorizations on $\C$, and functors between them.

\begin{defn}\label{reldiag}
\mbox{}
\begin{enumerate}[label=\alph*)]
\item 
Let $\reldiag$ denote the category generated by the graph
\[ \diagram
 \fullmoon   \ar[r]|-{\eta} &  \mathrm \Psi \ar@<1.5ex>[l]^-{\epsilon_1 } \ar@<-1.5ex>[l]_-{\epsilon_0 }
\enddiagram \]
and the equations $ \epsilon_0 \eta = \epsilon_1 \eta = 1_{ \fullmoon}$. A \emph{relation} on an object $X$ of a category $\C$ is a functor $R: \reldiag \to \C$ such that $R(\fullmoon) = X$.
\item Let $\relfactdiag $ denote the category generated by the graph
\[ \diagram
\fullmoon \ar[r]_-\lambda & \Phi \ar@/_2ex/[l]_\kappa \ar[r]_-\rho & \Sun
\enddiagram\]
and the equation $\kappa \lambda = 1_\fullmoon$. A \emph{relational factorization} of a morphism $f: X \to Y$ in a category $\C$ is a functor $P: \relfactdiag \to \C$ such that $P(\rho \lambda )= f$.
\item Let $\factdiag$ denote the category generated by the following graph and no equations.
\[ \diagram 
\fullmoon \ar[r]^\lambda & \Phi \ar[r]^\rho & \Sun
\enddiagram \]
A \emph{factorization} of a morphism $f: X \to Y$ in a category $\C$ is a functor $F: \factdiag \to \C$ such that $F(\rho \lambda) = f$. 

\end{enumerate}
\end{defn}

\begin{rem}
What we have defined above as a \emph{relation} could more descriptively be called an internal reflexive pseudo-relation. However, since all relations will be of this type, we will just call them relations.
\end{rem}

\begin{exmp} \label{hoeqex}
Consider an exponentiable object $I$ and morphisms $0,1: * \rightrightarrows I$ of $\C$ where $*$ is a terminal object. Let $!: I \to *$ denote the unique morphism to the terminal object.
\begin{enumerate}[label=\alph*)]
\item
On any object $X$ of $\C$, there is a relation whose image is the following diagram.
\[ \diagram
X  \ar[r]|-{X^!}&  X^I \ar@<1ex>[l]^-{X^1 } \ar@<-1ex>[l]_-{X^0 }
\enddiagram \]
\item
Consider a morphism $f: X \to Y$ in $\C$. Let $X \times_Y Y^I$ denote the pullback 
\[ \diagram
X \times_Y Y^I \pullback \ar[r] \ar[d] & Y^I \ar[d]^{Y^0}\\
X \ar[r]^f & Y
\enddiagram \]
 of $f:X \to Y$ and $Y^0: Y^I \to Y$. Then 
\[ \diagram
X \ar[rr]_-{1_X \times Y^! f} &&  \ar@/_2ex/[ll]_-{\pi_X} X \times_Y Y^I \ar[rr]^-{Y^1 \pi_{Y^I}} && Y
\enddiagram \]
is the image of a relational factorization of $f$, which we denote by $P: \reldiag \to \C$.
\item
We also obtain a factorization $F: \factdiag \to \C$ of $f$ whose image is depicted below.
\[ \diagram
X \ar[rr]^-{1_X \times Y^! f} &&   X \times_Y Y^I \ar[rr]^-{Y^1 \pi_{Y^I}} && Y
\enddiagram \]
\end{enumerate}
\end{exmp}

\begin{notn}
Let $\M$ denote the category generated by the graph $0 \to 1$. Then $\C^\M$ is the category of morphisms of $\C$. For a commutative square  in $\C$ as shown below, let $\langle \alpha, \beta \rangle : f \to g$ denote the morphism that this produces in $\C^\M$.
\[ \diagram
X \ar[r]^\alpha \ar[d]_f & W  \ar[d]^g \\
Y \ar[r]^\beta & Z
\enddiagram \]
\end{notn}

\begin{defn} \label{funcdmfs}
\mbox{}
\begin{enumerate}[label=\alph*)]
\item A \emph{functorial relation} on $\C$ is a section of the functor $\C^\fullmoon: \C^\reldiag \to \C$.  Let $\frel[\C]$ denote the category of such sections.
\item A \emph{functorial relational factorization} on $\C$ is a section of the functor $\C^{\rho\lambda}: \C^\relfactdiag \to \C^\M$.  Let $\frelfact[\C]$ denote the category of such sections.
\item A \emph{functorial factorization} on $\C$ is a section of the functor $\C^{\rho\lambda}: \C^\factdiag \to \C^\M$. Let $\ffact[\C]$ denote the category of such sections. 
\end{enumerate}
\end{defn}
\begin{exmp}\label{hoeqex2a}
Consider Example \ref{hoeqex}.
Since the relations, relational factorizations, and functorial factorizations given there are assembled from functors, these generate a functorial relation, a functorial relational factorization, and a functorial factorization on the category $\C$.
\end{exmp}

Now we take pains to consider variants of these concepts that are not functorial. This is because the $\Id$-types of Martin-L\"of type theory are not given functorially, and we aim to model these.

\begin{defn}
Let $\C$ and $\D$ be categories. An \emph{afunctor} $U: \C  \dashrightarrow \D$ consists of an object $U(X)$ of $\D$ for every object $X$ of $\C$ and a morphism $U(f): U(X) \to U(Y)$ for every morphism $f: X \to Y$ in $\C$.

Let $U,V: \C \dashrightarrow \D$ be afunctors.
An (unnatural) \emph{transformation} $\tau: U \dashrightarrow V$ consists of a morphism $\tau_X: U(X) \to V(X)$ in $\D$ for every $X$ in $\C$.

Categories and afunctors comprise a $1$-category $\aCat$ which contains the $1$-category $\Cat$ of categories and functors as a wide subcategory. We will use the fact that the unnatural transformations give every hom-set in $\aCat$ the structure of a category. Note, however, that this does not make $\aCat$ a $2$-category.
\end{defn}

\begin{defn}\label{afunctorialdefn}
\mbox{}
\begin{enumerate}[label=\alph*)]
\item A \emph{relation} on $\C$ is a section of the functor $\C^\fullmoon: \C^\reldiag \to \C$ in $\aCat$.  Let $\rel[\C]$ denote the category of such sections.
\item A \emph{relational factorization} on $\C$ is a section of the functor $\C^{\rho\lambda}: \C^\relfactdiag \to \C^\M$ in $\aCat$.  Let $\relfact[\C]$ denote the category of such sections.
\item A \emph{functorial factorization} on $\C$ is a section of the functor $\C^{\rho\lambda}: \C^\factdiag \to \C^\M$ in $\aCat$. Let $\fact[\C]$ denote the category of such sections. 
\end{enumerate}
\end{defn}

\begin{rem}
We are abusing terminology by speaking of relations, relational factorizations, and factorizations both on an object of a category and on the whole of a category.
\end{rem}

There are the following natural inclusions.
\[ \diagram
 \frel[\C]  \ar@{^{ (}->}[d] & \frelfact[\C] \ar@{^{ (}->}[d] &  \ffact[\C] \ar@{^{ (}->}[d]\\
  \rel[\C] &  \relfact[\C] & \fact[\C]
\enddiagram \]
We now describe the functors that fit horizontally into this diagram. 

There are functors 
\[\diagram
 \factdiag \ar[r]^{\iota} & \relfactdiag \ar[r]^\sigma & \reldiag
 \enddiagram \]
 where $\iota$ is the only injection $\factdiag \hookrightarrow \relfactdiag$ and $\sigma$ is the surjection sending $\kappa$ to $\epsilon_0$, $\rho$ to $\epsilon_1$, and $\eta$ to $\lambda$.
 
 The functor $\iota$ induces a functor $\iota_*: \relfact[\C]  \to \fact[\C]$ given by postcomposition with $\C^\iota: \C^\relfactdiag \to \C^\factdiag$. This restricts to a functor $\iota_*: \frelfact[\C]  \to \ffact[\C]$ making the right-hand square in the diagram below (\ref{diag:relrelfactfact}) commute. This functor takes a relational factorization on $\C$ to its underlying factorization on $\C$.

 \begin{equation} \label{diag:relrelfactfact} \diagram
 \frel[\C]  \ar@{^{ (}->}[d]  & \frelfact[\C] \ar@{^{ (}->}[d] \ar[r]^-{\iota_*} \ar[l]_{\sigma^*} &  \ffact[\C] \ar@{^{ (}->}[d]\\
   \rel[\C] &  \relfact[\C] \ar[r]^-{\iota_*}  \ar[l]_{\sigma^*} & \fact[\C]
\enddiagram 
\end{equation}

Let $!: \M \to *$ denote the unique morphism from the category $\M$ to the terminal category $*$. Since the following is a pullback diagram, pulling back along $\C^!$ produces a functor $\sigma^*: \relfact[\C] \to \rel[\C]$.
\[ \diagram 
\C^\reldiag  \ar[d]_{\C^\fullmoon} \pullback \ar@{^{ (}->}[r]^-{\C^\sigma}  & \C^\relfactdiag \ar[d]^{\C^{\rho\lambda}} \\
\C \ar@{^{ (}->}[r]^-{\C^!} & \C^\M \enddiagram \]
This restricts to a functor $\sigma^*: \frelfact[\C] \to \frel[\C]$ making the left-hand square in diagram (\ref{diag:relrelfactfact}) commute.

\begin{con}\label{con:mappingpath}
Consider an $R \in \rel[\C]$ and an $f:X \to Y$ in $\C$. We construct a relational factorization $\sigma_*(R)(f) \in \C^\relfactdiag$ of $f$. If we denote $R(Y)$ by the following diagram,
\[ \diagram
 Y   \ar[r]|-{\eta} &  {\Psi} Y \ar@<1.5ex>[l]^-{\epsilon_1 } \ar@<-1.5ex>[l]_-{\epsilon_0 }
\enddiagram \]
then we let $\sigma_*(R)(f)$ be the following diagram
\[ \diagram
X \ar[r]_-{1_X \times \eta f} & X \times_Y {\Psi}Y \ar@/_2ex/[l]_{\pi_X} \ar[r]_-{\epsilon_1 \pi_{{\Psi}Y}} & Y
\enddiagram\]
where $X \times_Y {\Psi}Y$ is the following pullback.
\[ \diagram 
X \times_Y {\Psi}Y \ar[r] \ar[d] \pullback & {\Psi}Y \ar[d]^{\epsilon_0} \\
X \ar[r]^f & Y
\enddiagram \]
\end{con}

\begin{lem}\label{lem:adj}
The construction above (\ref{con:mappingpath}) assembles into a functor $\sigma_*: \rel[\C] \to \relfact[\C]$ with functions 
\[ i: \hom_{\rel[\C]}(\sigma^* P, R) \leftrightarrows \hom_{\relfact[\C]}(P, \sigma_* R): j \]
for all $P \in \relfact[\C]$ and $R \in \rel[\C]$. The functor $\sigma_*$ restricts to a functor $\sigma_*: \frel[\C] \to \frelfact[\C]$, and the functions $i,j$ restrict to a bijection 
\[ i: \hom_{\frel[\C]}(\sigma^* P, R) \cong \hom_{\frelfact[\C]}(P, \sigma_* R): j \]
natural in $P$ and $R$, making $\sigma_*: \frel[\C] \to \frelfact[\C]$ right adjoint to $\sigma^*: \frelfact[\C] \to \frel[\C]$.
\end{lem}
\begin{rem}
The universal property of $\sigma_*$ can be interpreted as saying that $\sigma_* R$ is the right Kan extension of $\C^\sigma R: \C \to \C^\relfactdiag$ along $\C^!: \C \to \C^\M$ in $\Cat/ \C^\M$ \cite[Thm.~3.1.44]{Nor17}.
\end{rem}
\begin{proof}[Proof of Lemma~\ref{lem:adj}]
The functoriality of $\sigma^*$ is straightforward to check. We construct $i$ and $j$ and check that, when restricted to functorial relations and relational factorizations, they form a natural bijection.

Consider a $P \in \relfact[\C]$ which takes an $f:X \to Y $ in $\C$ to the diagram on the left below, and a $R \in \rel[\C]$ which takes $X \in \C$ to the diagram on the right below.
\[ \diagram
X \ar[r]_-{\lambda_f} & \Phi f \ar@/_2ex/[l]_{\kappa_f} \ar[r]_-{\rho_f} & Y
\enddiagram \ \ \ \ \ 
 \diagram
 X   \ar[r]|-{\eta_X} &  \Psi X \ar@<1.5ex>[l]^-{\epsilon_{1X} } \ar@<-1.5ex>[l]_-{\epsilon_{0X} }
\enddiagram 
\]

First, we construct a function $i: \hom_{\rel[\C]}(\sigma^* P, R) \to \hom_{\relfact[\C]}(P, \sigma_* R)$.
An element $\alpha \in \hom_{\rel[\C]}(\sigma^* P, R)$ has at each $X \in \C$, a component of the form shown on the left below.
\[
 \diagram
  X   \ar[r]|-{\lambda_{1_X}} \ar@{=}[d]&  \Phi 1_X \ar@<1.5ex>[l]^-{\rho_{1_X}} \ar@<-1.5ex>[l]_-{\kappa_{1_X} } \ar[d]^{a_X} & & X \ar@{=}[d] \ar[r]_-{\lambda_f} & \Phi f \ar@/_2ex/[l]_{\kappa_f} \ar[r]_-{\rho_f} \ar[d]|{\kappa_f \times a_Y \Phi \langle f,1_Y \rangle } & Y \ar@{=}[d]
  \\ 
 X   \ar[r]|-{\eta_X} &  \Psi X \ar@<1.5ex>[l]^-{\epsilon_{1X} } \ar@<-1.5ex>[l]_-{\epsilon_{0X} } & & X \ar[r]_-{1_X \times \eta_Y f} & X \times_Y \Psi Y \ar@/_2ex/[l]_{\pi_X} \ar[r]_-{\epsilon_1 \pi_{\Psi Y}} & Y
\enddiagram 
\]
Let $i\alpha \in \hom_{\relfact[\C]}(P, \sigma_* R)$ be the transformation with the component at each $f:X \to Y$ in $\C$ shown on the right above.

Now, we construct a function $j: \hom_{\relfact[\C]}(P, \sigma_* R) \to \hom_{\rel[\C]}(\sigma^* P, R)$. An element $\beta: \hom_{\relfact[\C]}(P, \sigma_* R) $ has at each $f: X \to Y$ in $\C$, a component of the form shown on the left below.
\[ \diagram
X \ar@{=}[d] \ar[r]_-{\lambda_f} & \Phi f \ar@/_2ex/[l]_{\kappa_f} \ar[r]_-{\rho_f} \ar[d]^{b_f } & Y \ar@{=}[d] & &   X   \ar[r]|-{\lambda_{1_X}} \ar@{=}[d]&  \Phi 1_X \ar@<1.5ex>[l]^-{\rho_{1_X}} \ar@<-1.5ex>[l]_-{\kappa_{1_X} } \ar[d]^{\pi_{\Psi X}b_{1_X}}
\\
X \ar[r]_-{1_X \times \eta_Y f} & X \times_Y \Psi Y \ar@/_2ex/[l]_{\pi_X} \ar[r]_-{\epsilon_1 \pi_{\Psi Y}} & Y & &  X   \ar[r]|-{\eta_X} &  \Psi X \ar@<1.5ex>[l]^-{\epsilon_{1X} } \ar@<-1.5ex>[l]_-{\epsilon_{0X} }
\enddiagram \tag{$*$}\]
Let $j \beta \in  \hom_{\rel[\C]}(\sigma^* P, R)$ be the transformation with the component at each $X \in \C$ shown on the right above.

Now suppose that $P$ is in $\frelfact[\C]$. To show that $j i \alpha = \alpha$, we show that the only nontrivial component of $\alpha_X$ for each $X \in \C$ is $a_X$. We calculate:
\begin{align*}
\pi_{\Psi X}(\kappa_{1_X} \times a_X \Phi \langle 1_X, 1_X \rangle) &= a_X \Phi \langle 1_X, 1_X \rangle\\
&= a_X 1_{\Phi 1_X} \\
&= a_X.
\end{align*}
Now to show that $ij \beta = \beta$, we calculate
\begin{align*}
\kappa_f \times \pi_{\Psi Y} {b_{1_Y}} \Phi \langle f,1_Y \rangle &=  \kappa_f \times  \pi_{\Psi Y} b_f \\
&=  b_f.
\end{align*}
Thus, $j = i^{-1}$.

Now we show that $j$ is natural in $P$ and $R$. Consider natural transformations $p: P' \to P$ and $r: R \to R'$. We want to show that the following diagram commutes.
\[ \diagram
 \hom_{\rel[\C]}(\sigma^* P, R) \ar[d]^{  r \circ - \circ \sigma^*p}& \hom_{\relfact[\C]}(P, \sigma_*R) \ar[l]_{j} \ar[d]^{\sigma_*r \circ - \circ p} \\ 
 \hom_{\rel[\C]}(\sigma^* P', R') & \hom_{\relfact[\C]}(P', \sigma_*R')\ar[l]_{j}
\enddiagram \] 
Denote the component of $p: P' \to P$ at a morphism $f: X \to Y$ in $\C$ by the diagram below on the left, and denote the component of $r: R \to R'$ at $X \in \C$ by the diagram below on the right.
\[ \diagram
X \ar@{=}[d] \ar[r]_-{\lambda'_f} & \Phi' f \ar@/_2ex/[l]_{\kappa'_f} \ar[r]_-{\rho'_f} \ar[d]|-{p_f} & Y \ar@{=}[d] & &   X   \ar[r]|-{\eta_{X}} \ar@{=}[d]&  \Psi X \ar@<1.5ex>[l]^-{\epsilon_{1X}} \ar@<-1.5ex>[l]_-{\epsilon_{0X} } \ar[d]^{r_X}
\\
X \ar[r]_-{\lambda_f} & \Phi f \ar@/_2ex/[l]_{\kappa_f} \ar[r]_-{\rho_f} & Y &  & X   \ar[r]|-{\eta'_X} &  \Psi' X \ar@<1.5ex>[l]^-{\epsilon'_{1X} } \ar@<-1.5ex>[l]_-{\epsilon'_{0X} }
\enddiagram 
 \]
For any $\beta \in \hom_{\relfact[\C]}(P, \sigma_*R)$ (with components as depicted in left-hand diagram of ($*$) above), $r \circ j\beta \circ \sigma^*p$ and $j(\sigma_*r \circ \beta \circ p)$ both have the following component at an object $X \in \C$.
\[
 \diagram
  X   \ar[r]|-{\lambda'_{1_X}} \ar@{=}[d]&  \Phi' 1_X \ar@<1.5ex>[l]^-{\rho'_{1_X}} \ar@<-1.5ex>[l]_-{\kappa'_{1_X} } \ar[d]^{r_Xb_{1_X}p_{1_x}}\\
 X   \ar[r]|-{\eta'_X} &  \Psi' X \ar@<1.5ex>[l]^-{\epsilon'_{1X} } \ar@<-1.5ex>[l]_-{\epsilon'_{0X} }
\enddiagram 
\]
Thus, $j$ is natural in $P$ and $R$.
\end{proof}

\begin{exmp}
Consider Example \ref{hoeqex2a}, and let $R$ denote the relation discussed there. Then $\sigma_* R$ is the relational factorization of that example.
\end{exmp}

\begin{con} \label{con:fact2relfact}
Consider an $F \in \fact[\C]$ and an $f: X \to Y$ in $\C$. We construct a factorization $\iota^*(F)(f) \in \C^\factdiag$ of $f$. Consider the morphism $1_X \times f: X \to X \times Y$. Denote $F(1_X \times f)$ by the following diagram.
\[ \diagram 
X \ar[r]^-{\lambda} &  \Phi(1_X \times f) \ar[r]^-\rho & X \times Y
\enddiagram \]
Then let $\iota^*(F)(f)$ be the following diagram.
\[ \diagram
X \ar[r]_-{\lambda} & \Phi(1_X \times f) \ar@/_2ex/[l]_-{\pi_X \rho} \ar[r]_-{\pi_Y \rho} & Y
\enddiagram\]
\end{con}
\begin{lem}
The construction above (\ref{con:fact2relfact}) assembles into a functor $\iota^*: \fact[\C] \to \relfact[\C]$ which restricts to a functor $\iota^*: \ffact[\C] \to \frelfact[\C]$.
\end{lem}
\begin{rem}
Though $\iota^*$ is not adjoint to $\iota_*$, $\iota_* \iota^*$ is a comonad on $\ffact[\C]$ \cite[Cor.~3.1.35]{Nor17}.
\end{rem}

Now we have described functors 
 \begin{equation} 
 \diagram
 \frel[\C]  \ar@{^{ (}->}[d]  \ar@<0.5ex>[r]^-{\sigma_*} & \frelfact[\C] \ar@{^{ (}->}[d] \ar@<0.5ex>[r]^-{\iota_*} \ar@<0.5ex>[l]^-{\sigma^*} &  \ffact[\C] \ar@{^{ (}->}[d] \ar@<0.5ex>[l]^-{\iota^*}\\
   \rel[\C] \ar@<0.5ex>[r]^-{\sigma_*} &  \relfact[\C]     \ar@<0.5ex>[r]^-{\iota_*} \ar@<0.5ex>[l]^-{\sigma^*} & \fact[\C]  \ar@<0.5ex>[l]^-{\iota^*}
\enddiagram 
\end{equation}
We are most interested in the functors between $\rel[\C]$ and $\fact[\C]$ (and between $\frel[\C]$ and $\ffact[\C]$). We abbreviate these as follows.
\begin{notn}\label{notn:functors}
Let $\reltofact$ denote $\iota_* \sigma _*$, and let $\facttorel$ denote $\sigma^* \iota^*$.\end{notn}

\subsection{Weak factorization structures and systems}
\label{sec:fibobj}
In this section, we discuss some fundamentals of weak factorization systems and the perspective on them that we take. Consider a factorization $F$ on $\C$. There are two functors $\lambda, \rho: \M \to \factdiag$ which take the non-identity morphism of $\M$ to $\lambda$ and $\rho$, respectively. Using these, we obtain afunctors $\lambda_F, \rho_F: \C^\M \to \C^\M$ from $F$. Then $\lambda_F$ is copointed and $\rho_F$ is pointed in the sense that for every $f: X\to Y$ of $\C$, there are the following morphisms in $\C^\M$.
\[ \diagram
 X \ar[d]_{\lambda_F f} \ar@{=}[r] & X \ar[d]^{f \ } && X \ar[d]_{  \ f} \ar[r]^-{\lambda_F f } & \ar[d]^{\rho_F f} \Phi_F f \\
 \Phi_Ff \ar[r]_-{\rho_F f}& Y && Y \ar@{=} [r]& Y 
\enddiagram \]
Let $\coalg{F}$ denote 
\[\{ f:X \to Y \text{ in } \C \ | \ \exists s : Y \to \Phi_F(f):  sf = \lambda_F(f) , \rho_F(f) s = 1_Y \},\]
 the class of coalgebras of the copointed endo-afunctor $(\lambda_F, \langle 1, \rho_F \rangle)$. This is the class of morphisms $f$ for which there is a lift in the square shown on the right above. 
Let $\alg{F}$ denote 
\[\{f:X \to Y \text{ in }  \C \ | \ \exists s :\Phi_F(f) \to X : s \lambda_F(f) = 1_X, fs = \rho_F f \},\]
 the class of algebras of the pointed endo-afunctor $( \rho_F, \langle \lambda_F, 1 \rangle)$. This is, dually, the class of morphisms $f$ for which there is a lift in the square above on the left. We will say that a morphism in $\coalg{F}$ has an $F$-coalgebra structure, and that a morphism in $\alg{F}$ has an $F$-algebra structure to simplify vocabulary. 

\begin{notn}
For any morphisms $\ell, r$ in $\C$, write $\ell \boxslash r$ if $\ell$ has the left lifting property against $r$.
For two collections $\mathsf L, \mathsf R$ of morphisms of $\C$, write $\mathsf L \boxslash \mathsf R$ if every morphism of $\mathsf L$ has the left-lifting property against every morphism of $\mathsf R$. Write $\mathsf L^\boxslash $ (and dually, ${^\boxslash \mathsf R}$) for the class of morphisms with the right lifting property against $\mathsf L$ (dually, the left lifting property against $\mathsf R$).
\end{notn}

\begin{prop}\label{coalgalglift} 
	Consider a factorization $F$ on $\C$. Then $\coalg{F} \boxslash \alg{F}$.
\end{prop}
\begin{proof}
This appears in Corollary 2.7 of \cite{GT06}. Though they consider only functorial factorizations, their argument works here without modification.
\end{proof}
\begin{defn}
A \emph{weak factorization structure} on $\C$ is a factorization $F$ on $\C$ such that for every morphism $f$ of $\C$, $\lambda_F(f) \in \coalg{F}$ and $\rho_F(f) \in \alg{F}$. An \emph{algebraic weak factorization structure} on $\C$ is a functorial factorization $F$ on $\C$ such that the copointed endofunctor $\lambda_F$ underlies a comonad on $\C^\M$ and the pointed endofunctor $\rho_F$ underlies a monad on $\C^\M$. 

Let $\WFS[\C]$ denote the full subcategory of $\fact[\C]$ spanned by those objects which are weak factorization structures.
\end{defn}
\begin{notn}
For any category $\D$, let $|\D|$ denote the preordered truncation: the preorder (viewed as a category) which has the same objects as $\D$ and a morphism $X \to Y$ when there is a morphism $X \to Y$ in $\D$.

For any object $X$ of $\D$, we will let $[X]$ denote the isomorphism class of $X$ in $|\D|$, and we will say that two objects $X$ and $Y$ of $\D$ are \emph{equivalent} if they are isomorphic in $|\D|$.
\end{notn}
\begin{prop}\label{equivwfs}
The isomorphism classes of $|\WFS[\C]|$ are the weak factorization systems on $\C$.
\end{prop}
\begin{proof}
We show that there is a function $q$ from the objects of $\WFS[\C]$ to the weak factorization systems on $\C$ which is surjective and whose fibers are the isomorphism classes of $|\WFS[\C]|$.

By Theorem 2.4(2) of \cite{RT02}, for any $F \in \WFS[\C]$, $(\coalg{F}, \alg{F})$ is a weak factorization system on $\C$ with factorization given by $F$. Let $q$ denote the function which maps a weak factorization structure $F$ to the weak factorization system $(\coalg{F}, \alg{F})$.

By Proposition 5.1 of \cite{RT02}, for any $F, G \in \WFS[\C]$, we have that \[(\coalg{F} , \alg{F}) = (\coalg{G} , \alg{G})\] if and only if there are morphisms $F \leftrightarrows G$. Therefore, the fibers of $q$ are the isomorphism classes of $|\WFS[\C]|$.

Consider a weak factorization system $(\mathsf L, \mathsf R)$. There exists a factorization of each morphism $f: X \to Y$ in $\C$ which we can denote by the diagram on the left below.
\[ \diagram
X \ar[r]^-{\lambda(f)} & Mf \ar[r]^-{\rho(f)} & Y
\enddiagram \hspace{40pt}
\diagram
X \ar[d]_-{\lambda(f)} \ar[r]^-{\lambda(g) \alpha }& Mg \ar[d]^{\rho(g)} \\
Mf \ar[r]_-{\beta \rho(f)} \ar@{-->}[ur]& Z \enddiagram\]
For each morphism $g: W \to Z$ and $\langle \alpha, \beta \rangle: f \to g$, we can obtain a morphism $M\langle \alpha, \beta \rangle: Mf \to Mg$ by considering the lifting diagram above on the right. This assembles into a factorization, say $F$, on $\C$. By Theorem 2.4(1) of \cite{RT02}, $\mathsf L = \coalg{F}$ and $\mathsf R = \alg{F}$. Thus, $q$ is surjective.

(Again, though only functorial factorizations are considered in \cite{RT02}, their proofs of these results work here without modification.)
\end{proof}
\subsection{Display map categories}\label{sec:dmc}
Now we define what we consider in this paper to be a categorical interpretation of dependent type theory.
\begin{defn}
A class $\D$ of morphisms of $\C$ forms a \emph{display map category} $(\C, \D)$ when the following hold:
\begin{enumerate}
\item $\D$ contains every isomorphism;
\item $\D$ contains every morphism whose codomain is a terminal object; and
\item $\D$ is stable under pullback.
\end{enumerate}
We call the elements of $\D$ \emph{display maps}.
\end{defn}
The notion of {display map category} is closely related to others in the literature \cite{Tay99,Shu15,Joy17}. There is a careful comparison of this notion and of the types described below with others in the literature in \cite{Nor19}. 
\begin{defn}
Let $(\mathsf L, \mathsf R)$ be a weak factorization system on $\C$. We say that an object $X$ of $\C$ is \emph{fibrant} if every morphism from $X$ to a terminal object is in $\mathsf R$.
\end{defn}

\begin{exmp} Let $(\mathsf L, \mathsf R)$ be a weak factorization system on $\C$ in which all objects are fibrant. Since right classes of weak factorization systems always contain all isomorphisms and are stable under pullbacks \cite[Prop.~14.1.8]{MP12}, $(\C, \mathsf R)$ is a display map category.
\end{exmp}

\begin{defn}\label{sigmatypes}
A {display map category $(\C, \D)$} \emph{models $\Sigma$-types} if $\D$ is closed under composition.  We call a composition $gf$ of display maps a \emph{$\Sigma$-type} and sometimes denote it by $\Sigma_g f$.
\end{defn}

\begin{exmp} \label{wfssigma} Let $(\mathsf L, \mathsf R)$ be a weak factorization system on $\C$ in which all objects are fibrant. Since right classes of weak factorization systems are always closed under composition \cite[Prop.~14.1.8]{MP12}, $(\C, \mathsf R)$ models $\Sigma$-types.
\end{exmp}

Weak factorization systems are not only examples of display map categories but are also induced by ones which model $\Id$-types. 

\begin{defn} \label{Idtypes}
Consider a {display map category} ($\C, \D$) which models $\Sigma$-types, and consider a display map $f:X \to Y$. An \emph{identity type of $f$} is a relation on $f$ in the slice $\C / Y$
\begin{equation} \diagram
X \ar[dr]_f \ar[r]^-{r_f} & \Id(f) \ar[d]^{\iota_f} \ar[r]^-{\epsilon_f} & X \times_Y X \ar[dl]^{f \times f} \\
& Y 
\enddiagram \label{eq:idfact}
\end{equation}
such that 
 \begin{enumerate}
 \item $\epsilon_f$ is in $ \D$ and
 \item for every morphism $\alpha: A \to X$ in $\C$, the pullback $ \alpha^* r_f $, as shown below, is in $^\boxslash{\D}$ for $i = 0,1$.
\begin{equation}
\diagram 
&\alpha^*\Id(f) \ar[rrr] \ar[dd]|!{"2,1";"2,4"}\hole &&&{\Id(f)} \ar[dd]^{\pi_i \epsilon_f} \\
A \ar[rrr] \ar@{=}[dr] \ar[ur]^{\alpha^* r_f} &&& X \ar@{=}[dr] \ar[ur]^{r_f} & \\
& A \ar[rrr]^\alpha && & X
\enddiagram 
\end{equation}
\end{enumerate}
We will call the morphism $\iota_f: \Id(f) \to Y$ in Diagram (\ref{eq:idfact}) the \emph{$\Id$-type of $f$} in $ \C /Y$. 
\end{defn}
Given a morphism $\alpha: f \to g$ in $\C/Y$ and identity types of $f$ and $g$, there is a natural transformation between these two relations because $r_f \in {^\boxslash{\D}}$ and $\epsilon_g \in {\D}$.
\[ \xymatrix@C=7ex{
f \ar[d]^-{r_f} \ar[r]^{r_g \alpha} & \iota(g) \ar[d]^{\epsilon_g} \\
\iota(f) \ar[r]_-{(\alpha \times \alpha) \epsilon_f} \ar@{-->}[ur] & g \times g
} \]
Thus, in what follows, when we assume that every object of $\C$ has an $\Id$-type, we will assume that there is a relation on $\C$ which specifies these $\Id$-types.

\begin{defn}
Consider a display map category $(\C, \D)$ which models $\Sigma$-types. 
If there is a relation $I$ on $\C$ for which $I(X)$ is an identity type on $X$ for each object $X$ of $\C$, then we say that $I$ is a \emph{model of $\Id$-types of objects in $(\C, \D)$}  or just that
 \emph{$(\C, \D)$ models $\Id$-types of objects.} 
If there is a relation $I_Y$ on each slice $\C/Y$ for which $I_Y(f)$ is an identity type of $f$ for each display map $f:X \to Y$ of $\C$, then we say that the collection $\{ I_Y \}_{Y \in \C}$ is a \emph{model of $\Id$-types in $(\C, \D)$} or just that \emph{$(\C, \D)$ models $\Id$-types.} If these relations are functorial, then we call the $\Id$-types \emph{functorial}.
\end{defn}

For any display map category $(\C, \D)$ which models $\Sigma$-types and has a model $I$ of $\Id$-types of objects, the factorization $\reltofact(I)$ (with $\reltofact$ as defined in Notation \ref{notn:functors}) is a weak factorization structure, and its underlying weak factorization system $[\reltofact(I)]$ is $({^\boxslash  \D}, (^\boxslash  \D)^\boxslash)$ \cite[Thm.~2.8]{Emm14}. Our goal in this paper is to understand which weak factorization systems arise in this way. 

Note that $ (^\boxslash \D)^\boxslash$ is the retract closure of $ \D$ and so in particular contains $\D$. Thus, to decide whether a weak factorization system $(\mathsf L, \mathsf R)$ on a category $\C$ does arise from a model of $\Sigma$- and $\Id$-types, it seems that we might have to check whether $(\C, \D)$ is a display map category modeling $\Sigma$- and $\Id$-types for all $\D$ whose retract closure is $\mathsf R$. However, we showed in \cite[Thm.~5.12]{Nor19} that if $(\C, \D)$ is a display map category modeling $\Sigma$- and functorial $\Id$-types, then $(\C, (^\boxslash \D)^\boxslash)$ is a display map category modeling $\Sigma$- and functorial $\Id$-types (and if $(\C, \D)$ modeled $\Pi$-types, then so does $(\C, (^\boxslash \D)^\boxslash)$). Thus, to decide whether a weak factorization system $(\mathsf L, \mathsf R)$ does arise from a model of $\Sigma$- and functorial $\Id$-types, we only have to decide if $(\C, \mathsf R)$ is a display map category modeling functorial $\Sigma$- and $\Id$-types. In this paper, we show how one can decide such a thing, and we will show that we can also drop the requirement of functoriality on $\Id$-types.

In particular, suppose that we want to decide whether a weak factorization system $(\mathsf L, \mathsf R)$ arises as $[\reltofact(I_*)]$ from a model $I$ of (functorial) $\Id$-types in a display map category $(\C,\D)$ (where $*$ is a terminal object of $\C$, so that $I_*$ is a model of $\Id$-types of objects). If it does, $(\mathsf L, \mathsf R)$ itself has a model $J$ of (functorial) $\Id$-types. Whenever $(\mathsf L, \mathsf R)$ itself has a model $J$ of $\Id$-types, we have $[\reltofact(J_*)] = ({^\boxslash \mathsf R}, (^\boxslash \mathsf R)^\boxslash) = (\mathsf L, \mathsf R)$. We will also see in Corollary \ref{dropfunc} that $(\mathsf L, \mathsf R)$ models $\Id$-types if and only if it models $\Id$-types of objects. Thus, a weak factorization system arises from a model of $\Id$-types if and only if it models $\Id$-types of objects; we call such a model an \emph{$\Id$-presentation} of the weak factorization system.

\begin{defn}\label{relidtypes2}
We say that a relation $I$ on $\C$ is an \emph{$\Id$-presentation of the weak factorization system $[\reltofact(I)]$} if
 the factorization $\reltofact(I) $ is a weak factorization structure and $R$ is a model of $\Id$-types of objects in $(\C, \alg{{\reltofact(R)}})$.
\end{defn}

Note that for any relation $I$ which generates a weak factorization structure $\reltofact(I)$, all objects are fibrant. Consider any object $X$ in $\C$. The solution shown in the following lifting problem is a $\rho_{\reltofact(I)}$-algebra structure for $!: X \to *$.
\[ \diagram 
X \ar@{=}[r] \ar[d]_{1_X \times \eta_* !} & X \ar[d] \\
X \times \Psi(*) \ar[r] \ar[ur]^{\pi_X}& *
\enddiagram \]
Thus, in the definition (\ref{relidtypes2}) above, $(\C, \alg{{\reltofact(R)}})$ is in fact a display map category, and so it makes sense to talk of models of $\Id$-types of objects in it. We record this fact here.
\begin{prop}\label{relfact2fib}
Let $R$ be a relation on $\C$ which produces a weak factorization structure $\reltofact(R)$. Then every object is fibrant in $[ \reltofact(R) ]$.
\end{prop}

If a weak factorization system is going to have an $\Id$-presentation, then all objects must be fibrant in it. The only other condition that we will find we need to place on a weak factorization system to ensure that it does have an $\Id$-presentation is often called the \emph{Frobenius property} \cite{BG12} and is closely related to modeling $\Pi$-types.

\begin{defn}\label{pilem}
A display map category $(\C, \D)$ \emph{models pre-$\Pi$-types} if for every $g: W \to X$ and $f: X \to Y$ in $\D$, there is a morphism $\Pi_f g$ with codomain $Y$ satisfying the universal property
\[i: \C/X(f^* y, g) \cong \C/Y(y,\Pi_f g ) \]
natural in $y$.
Call the map $\Pi_f g$ a \emph{pre-$\Pi$-type}.

A display map category $(\C, \D)$ \emph{models $\Pi$-types} if it models pre-$\Pi$-types and each pre-$\Pi$-type is a display map.
\end{defn}
\begin{prop}\label{pre2pitypes}
Consider a weak factorization system $(\mathsf L, \mathsf R)$ on $\C$ in which all objects of $\C$ are fibrant and which models pre-$\Pi$-types.

Then $(\C, \mathsf R)$ models $\Pi$-types if and only if $\mathsf L$ is stable under pullback along $\mathsf R$.
\end{prop}

This proposition is very similar to the fact that a left adjoint between two categories with weak factorization systems preserves the left maps \cite[\S16.2]{MP12} if and only if the right adjoint preserves the right maps. Our setting, however, is a bit more convoluted.

\begin{proof}[Proof of Proposition~\ref{pre2pitypes}]
Suppose that $(\C, \mathsf R)$ models $\Pi$-types. Let $i$ denote the bijection
\[i: \C / X(f^* y, g) \cong  \C / Y(y, \Pi_f g) \]
of Definition \ref{pilem}, and consider a morphism $\ell$ of $\mathsf L$ and a morphism $f$ of $\mathsf R$ such that $\cod \ell = \cod f$. To show that $f^* \ell$ is in $\mathsf L$, we must show that for any factorization $\lambda(f^* \ell), \rho(f^* \ell)$ of $f^* \ell$ for which $\lambda(f^* \ell) \in \mathsf L$, $\rho(f^* \ell) \in \mathsf R$, the lifting problem shown on the left below has a solution.
\[ \diagram
\ar[d]_{f^* \ell}{\ }  \ar[r]^{\lambda(f^* \ell)}&{\ } \ar[d]^{\rho(f^* \ell)} & & \ar[d]_{ \ell}{\ }  \ar[r]^{i(\lambda(f^* \ell))}&{\ } \ar[d]^{\Pi_f \rho(f^* \ell)}
\\
\ar@{=}[r] \ar@{-->}[ur] {\ }&{\ } & & \ar@{=}[r] \ar@{-->}[ur]^{\sigma} {\ }&{\ }
\enddiagram \]
Consider the lifting problem shown on the right above. It is the transpose of the above lifting problem under $i$.
It has a solution $\sigma$ since $\ell $ is in $ \mathsf L$ and $\Pi_f \rho(f^* \ell) $ is in $\mathsf R$. Then $i^{-1}( \sigma)$ gives us a solution to our original lifting problem.

Now suppose that $\mathsf L$ is stable under pullback along $\mathsf R$. We need to show that $\Pi_f g$ is in $\mathsf R$. 
The morphism $\Pi_f g$ is in $\mathsf R$ if and only if for any factorization $\lambda(\Pi_f g), \rho(\Pi_f g)$ of $\Pi_f g$ for which $\lambda(\Pi_f g) \in \mathsf L$ and $\rho(\Pi_f g) \in \mathsf R$, 
 there is a solution to the lifting problem shown below on the left.
\[ \diagram
\ar[d]_{\lambda(\Pi_f g)}{\ }  \ar@{=}[r]&{\ } \ar[d]^{\Pi_f g} & & &  \ar[d]_{f^* \lambda(\Pi_f g)}{\ }  \ar[r]^{i^{-1}(1)}&{\ } \ar[d]^{g} 
 \\
\ar[r]_{\rho(\Pi_f g)} \ar@{-->}[ur] {\ }&{\ } & & &  \ar[r]_{f^* \rho(\Pi_f g)} \ar@{-->}[ur]^{\sigma} {\ }&{\ }
\enddiagram \]
Consider the lifting problem on the right above. It is the transpose of the original lifting problem under $i^{-1}$.
Since $f^* \lambda(\Pi_f g)$ is in $\mathsf L$ and $g$ is in $\mathsf R$, there is a solution $\sigma$ to this lifting problem. 
Then $i (\sigma)$ is a solution to the original lifting problem.
\end{proof}

\begin{defn}
A weak factorization system $(\mathsf L, \mathsf R)$ on a finitely complete category $\C$ is \emph{type-theoretic} if it has the following two properties:
\begin{enumerate}
\item all objects are fibrant, and
\item $\mathsf L$ is stable under pullback along $ \mathsf R$.
\end{enumerate}
A weak factorization structure $F$ is \emph{type-theoretic} if $[F]$ is.
\end{defn}



\subsection{Summary}\label{sec:technicalsummary}

We have described the following diagram of categories of relations and factorizations.
\[
 \diagram
 \frel[\C]  \ar@{^{ (}->}[d]  \ar@<0.5ex>[r]^-{\reltofact} &  \ffact[\C] \ar@{^{ (}->}[d] \ar@<0.5ex>[l]^-{\facttorel}\\
   \rel[\C] \ar@<0.5ex>[r]^-{\reltofact}  & \fact[\C]  \ar@<0.5ex>[l]^-{\facttorel}
\enddiagram 
\]

We are interested in the relationship between type-theoretic weak factorization structures and $\Id$-presentations of weak factorization systems. The former are a kind of factorization, so they naturally form a full subcategory of $ \fact[\C]$. The latter are a kind of relation, so they naturally form a full subcategory of $ \rel[\C]$.

\begin{defn}\label{idprescatdef}
Let $\ttWFS[\C]$ be the full subcategory of $\fact[\C]$ spanned by the type-theoretic weak factorization structures on $\C$, and similarly, let $\fttWFS[\C]$ be the full subcategory of $\ffact[\C]$ spanned by the type-theoretic weak factorization structures on $\C$.
\end{defn}
\begin{defn}\label{ttwfscatdef}
Let $\IdPres[\C]$ denote the full subcategory of $\rel[\C]$ spanned by those relations which are $\Id$-presentations, and similarly let $\fIdPres[\C]$ denote the full subcategory of $\frel[\C]$ spanned by the $\Id$-presentations.
\end{defn}

Then we are interested in what relationship the subcategories $\ttWFS[\C]$ and $\IdPres[\C]$ have in the following diagram.
\[
 \diagram
 \fIdPres[\C] \ar@{^{ (}->}[r] \ar@{^{ (}->}[d]  & \frel[\C]  \ar@{^{ (}->}[d]  \ar@<0.5ex>[r]^-{\reltofact} &  \ffact[\C] \ar@{^{ (}->}[d] \ar@<0.5ex>[l]^-{\facttorel} & \fttWFS[\C] \ar@{_{ (}->}[l]  \ar@{^{ (}->}[d]  \\
 \IdPres[\C]  \ar@{^{ (}->}[r] &   \rel[\C] \ar@<0.5ex>[r]^-{\reltofact}  & \fact[\C]  \ar@<0.5ex>[l]^-{\facttorel}&   \ttWFS[\C] \ar@{_{ (}->}[l]
\enddiagram 
\]
In the next sections, we show that $\reltofact, \facttorel$ restrict to functors between $\IdPres[\C]$ and $\ttWFS[\C]$ and constitute an equivalence between them under the proset truncation.

\section{Type-theoretic weak factorization systems from Moore relations}
\label{sec:mrs}

In this section, we consider a finitely complete category $\C$ and a relation $ R$ on $\C$. In the first subsection, we describe structure on $ R$ which will make $\reltofact( R)$ a type-theoretic, algebraic weak factorization structure. We call this a \emph{strict Moore relation structure}. In the second subsection we describe structure on $ R$ which will make $\reltofact( R)$ a type-theoretic weak factorization structure. We call this structure a \emph{Moore relation structure}. 

In Section \ref{sec:idpresandmrs}, we will show that any relation is an $\Id$-presentation of a weak factorization system if and only if it has a Moore relation structure. Then the full subcategory of $\rel[\C]$ spanned by Moore relation systems will coincide with $\IdPres[\C]$. We originally defined the subcategory $\IdPres[\C]$ by referencing the functor $\reltofact : \rel[\C] \to \fact[\C]$. The description of Moore relation structures which follows describes this subcategory more directly, without making reference to $\reltofact $. We will need this direct description to connect the category $\IdPres[\C]$ with the category $\ttWFS[\C]$, the goal of this chapter.

We are mostly interested in the (non-strict) Moore relation structures since these correspond to $\Id$-presentations. However, first we describe strict Moore relation structures. As mentioned in the introduction, these have already been investigated in \cite{BG12}. We mention these first because they have many natural examples, and are thus more readily understandable. By contrast, most examples of non-strict Moore relation structures will come from the equivalence between them and type-theoretic weak factorization systems.

\subsection{Strict Moore relation systems}
\label{sec:smrs}

In this subsection, we consider a functorial relation $  R$ which preserves pullbacks. For any object $X$ in $\C$, denote the image of $  R X$ by 
\[ \diagram
 X   \ar[r]|-{\eta} &   \Psi X \ar@<1.5ex>[l]^-{\epsilon_{1} } \ar@<-1.5ex>[l]_-{\epsilon_{0} }
\enddiagram .\]
Note that the requirement that $ R$ preserves pullbacks is equivalent to the requirement that $\Psi$ does.

For any morphism $f: X \to Y$ of $\C$, denote the factorization $\reltofact(  R)f$ by the following diagram.
\[ \diagram
X \ar[rr]^-{\lambda_f} && Mf  \ar[rr]^-{\rho_f} && Y
\enddiagram \]
Recall that $\lambda$ is a copointed endofunctor on $\C^\M$, and $\rho$ is a pointed endofunctor on $\C^\M$.

In this section, we discuss the structure on $ R$ that will produce a comonad structure on the copointed endofunctor $\lambda$ and a monad structure on the pointed endofunctor $\rho$.
\subsubsection{Strictly transitive functorial relations}

\begin{defn}
Say that a functorial relation $ R$ on $\C$ is \emph{strictly transitive} if there exists a natural transformation $\mu$ with components
\[\mu_X : \pullb{\Psi X}{\Psi X}{\epsilon_1}{\epsilon_0} \to \Psi X\]
for each object $X$ in $\C$
such that:
\begin{enumerate}
\item $\epsilon_i \mu = \epsilon_i \pi_i$ for $i = 0,1$ in the diagram below.
\begin{align}\diagram
  \pullb{\Psi  }{\Psi  }{\epsilon_1}{\epsilon_0} \ar@<0.5ex>[d]^-{\epsilon_1 \pi_1 } \ar@<-0.5ex>[d]_-{\epsilon_0 \pi_0 } \ar[rr]^{\mu} & &  \Psi  \ar@<0.5ex>[d]^-{\epsilon_1 } \ar@<-0.5ex>[d]_-{\epsilon_0 }
 \\  1_\C    \ar@{=}[rr] & & 1_\C  
\enddiagram \label{mulift} \end{align}
\item$(1_\C, \Psi , \epsilon_0, \epsilon_1, \eta, \mu)$ is an internal category in $[\C, \C]$; that is, the following diagrams commute.
\begin{align}
\diagram
 \Psi  \ar[r]^-{\eta \times 1} \ar@{=}[dr] & \pullb{\Psi }{\Psi }{\epsilon_1}{\epsilon_0} \ar[d]^{\mu} & \Psi  \ar[l]_-{1 \times \eta}\ar@{=}[dl]  \\
& \Psi 
\enddiagram 
& & \diagram
\pullb{\Psi }{\Psi }{\epsilon_1}{\epsilon_0} \pullb{}{\Psi }{\epsilon_1}{\epsilon_0} \ar[d]^{\mu \times 1} \ar[rr]^{1 \times \mu} & &  \pullb{\Psi }{\Psi }{\epsilon_1}{\epsilon_0} \ar[d]^\mu\\
\pullb{\Psi }{\Psi }{\epsilon_1}{\epsilon_0} \ar[rr]^{\mu} & &  \Psi 
\enddiagram
\label{intcat2}
\end{align}
\end{enumerate}
\end{defn}

Note that if $ R$ is a monic relation, then the existence of $\mu$ with the commutativity of the diagram in \eqref{mulift} says that the relation $ R(X)$ on each object $X$ of $\C$ is transitive, and the commutativity of the diagrams in \eqref{intcat2} is automatic. Thus, the notion of transitivity here is a generalization of the usual one.

\begin{exmp}\label{hoeqex3}
Consider the relation which takes any object $X$ in $\C$ to the following diagram 
\[ \diagram
 X   \ar[r]|-{X^!} &  \mathrm X^I \ar@<1.5ex>[l]^-{X^1 } \ar@<-1.5ex>[l]_-{X^0 }
\enddiagram \]
as in Example \ref{hoeqex}.

Suppose that there a morphism $m$ making the following diagrams commute.
\[ 
\diagram
* \ar[d]^i \ar[r]^{i}  & I \ar[d]^{\iota_i}  \\
I \ar[r]^-m &  I {_1+_0} I  
\enddiagram
\ \ \ \ \ \ \ 
\diagram
&& I\\
I \ar[r]^-m \ar@/^/@{=}[urr]  \ar@/_/@{=}[drr] & I {_1+_0} I  \ar[ur]_{! + 1_I}  \ar[dr]^{1_I + !}\\
&& I
\enddiagram 
\ \ \ \ \ \ \ 
\diagram 
I \ar[r]^m \ar[d]^m & I {_1+_0} I  \ar[d]^{m {_1+_0} I   } \\
I {_1+_0} I \ar[r]^-{I {_1+_0}   m}   & I {_1+_0}   I {_1+_0} I  
\enddiagram\]
Then taking $X^m: \pullb{X^I}{X^I}{\epsilon_1}{\epsilon_0} \to X^I$ for $\mu_X$ makes this relation strictly transitive.

For example, in the category $\Cat$, there is such an $m$ when $I$ is $\M$ (i.e., the category generated by the graph $0 \to 1$) or the groupoid generated by the graph $0 \to 1$.
\end{exmp}

\begin{exmp} \label{moorepaths1}
Consider the category $\T$ of topological spaces. Let $\R^+$ denote the non-negative reals, and let $\Gamma X$ denote the subspace of $X^{\R^+} \times \R^+$ consisting of pairs $(p, r)$ such that $p$ is constant on $[r,\infty)$. This is called the space of \emph{Moore paths} in $X$, and it is functorial in $X$. We think of this as the space of paths in $X$ of finite length.

There is a natural transformation $c: X \to \Gamma X$ which maps $x \in X$ to the constant path of length $0$ at $x$. There are natural transformations $ev_0, ev_\infty: \Gamma X \to X$ which map a pair $(p, r)$ to $p(0)$ and $p(r)$, respectively. These assemble into a functorial relation $G $ on $\T$.

There is also a natural transformation $\mu_X: \pullb{\Gamma X}{\Gamma X}{ev_\infty}{ev_0} \to \Gamma X$ which maps two paths to their concatenation. To be precise, it takes a pair $((p,r),(p',t')) $ such that $p(r) = p'(0)$ to the pair $(q,s)$ where $s = r + r'$, $q|_{[0,r]} = p|_{[0,r]}$, and $q(x)|_{[r,\infty)} = p'(x-r)$. This makes $G $ a strictly transitive functorial relation.
\end{exmp}

\begin{prop}\label{rightalg} Let $ R$ be a strictly transitive functorial relation on $\C$.
Then the functor $\rho: \C^\M \to \C^\M$ underlies a monad on $\C^{\mathbb 2}$ with unit and multiplication components at an object $f: X \to Y$ in $\C^{\mathbb 2}$ given by the following diagrams
\[\diagram 
X \ar[d]^f \ar[r]^-{\lambda_f} & Mf \ar[d]^{\rho_f}\\
Y \ar@{=}[r] & Y
\enddiagram 
 \ \ \ \diagram
 M\rho_f  \ar[r]^-{1 \times \mu} \ar[d]^{\rho^2_f}& Mf \ar[d]^{\rho_f}
\\
Y \ar@{=}[r] & Y
 \enddiagram\]
where $1 \times \mu:  M\rho_f  \to Mf $ is the morphism 
\[1_X \times \mu_Y:   \pullb{X}{\Psi  Y}{f}{\epsilon_0}  \pullb{}{\Psi  Y}{\epsilon_1}{\epsilon_0} \longrightarrow  \pullb{X}{\Psi  Y}{f}{\epsilon_0} . \]
\end{prop}

\begin{proof}
We have already seen that the unit square above commutes. The commutativity of the multiplication square above follows from the commutativity of \eqref{mulift}.

The following diagram displays the unit axioms for the monad.
\[ \diagram
\pullb{X}{\Psi  Y}{f}{\epsilon_0} \ar@{=}[dr] \ar[r]^-{1 \times 1 \times \eta} &  \pullb{X}{\Psi  Y}{f}{\epsilon_0}  \pullb{}{\Psi  Y}{\epsilon_1}{\epsilon_0}   \ar[d]^{1 \times \mu} & & \pullb{X}{\Psi  Y}{f}{\epsilon_0}  \ar[ll]_-{ 1 \times \eta f \times 1} \ar@{=}[dll] \\
& \pullb{X}{\Psi  Y}{f}{\epsilon_0}
\enddiagram \]
Its commutativity follows from that of the left-hand diagram in \eqref{intcat2}.

This diagram displays the associativity axiom for the monad.
\[ \diagram 
\pullb{X}{\Psi  Y}{f}{\epsilon_0}  \pullb{}{\Psi  Y}{\epsilon_1}{\epsilon_0} \pullb{}{\Psi  Y}{\epsilon_1}{\epsilon_0} 
\ar[r]^-{1 \times 1 \times \mu} \ar[d]^{1 \times \mu \times 1}
& \pullb{X}{\Psi  Y}{f}{\epsilon_0}  \pullb{}{\Psi  Y}{\epsilon_1}{\epsilon_0}   \ar[d]^{1 \times \mu} 
\\ \pullb{X}{\Psi  Y}{f}{\epsilon_0}  \pullb{}{\Psi  Y}{\epsilon_1}{\epsilon_0}  \ar[r]^-{1 \times \mu} 
& \pullb{X}{\Psi  Y}{f}{\epsilon_0}   
\enddiagram \]
Its commutativity follows from that of the right-hand diagram in \eqref{intcat2}.
\end{proof}

\subsubsection{Strictly homotopical functorial relations}

\begin{defn}
Say that a functorial relation $ R: \C \to \C^\reldiag$ is \emph{strictly homotopical} if there exist natural transformations $\delta$ and $\tau$ with components
\begin{align*}
\delta_X &: \Psi  X \to \Psi ^2 X \\
\tau_X &: X \times \Psi (*) \to \Psi  X \\
\end{align*}
for each object $X$ of $\C$
such that:
\begin{enumerate}
\item $\eta \eta = \delta \eta$ and $\eta \epsilon_0 = \epsilon_0 \delta$ in the following diagram.
\begin{align}\diagram
\Psi   \ar@<-1.5ex>[d]_-{\epsilon_0 } \ar[rr]^{\delta} & &  \Psi ^2  \ar@<-1.5ex>[d]_-{\epsilon_0 }
 \\  1_\C   \ar[u]|-{\eta} \ar[rr]^\eta & & \Psi    \ar[u]|-{\eta}
\enddiagram \label{deltalift}\end{align}
\item $\epsilon_i \tau = \pi$ for $i = 0,1$ and $\tau (1 \times \eta) = \eta$ in the following diagram.
\begin{align}\diagram
1_\C \times \Psi (*) \ar@<1.5ex>[d]^-{\pi} \ar@<-1.5ex>[d]_-{\pi} \ar[rr]^-{\tau} & &  \Psi  \ar@<1.5ex>[d]^-{\epsilon_1 } \ar@<-1.5ex>[d]_-{\epsilon_0 }
 \\  1_\C   \ar[u]|-{1 \times \eta} \ar@{=}[rr] & & 1_\C   \ar[u]|-{\eta}
\enddiagram \label{taulift}\end{align}
\item $(\Psi , \epsilon_1, \delta)$ is a comonad on $\C$; that is, the following diagrams commute.
\begin{align}
\diagram
& \Psi  \ar[d]^\delta \ar@{=}[dr] \ar@{=}[dl] \\
\Psi  & \Psi ^2 \ar[l]^{\epsilon_1} \ar[r]_{\Psi  \epsilon_1} & \Psi 
 \enddiagram & &
 \diagram
\Psi  \ar[d]^\delta \ar[rr]^\delta & &  \Psi ^2 \ar[d]^\delta\\
\Psi ^2 \ar[rr]^{\Psi \delta} & &  \Psi ^3
\enddiagram
\label{comonad}
\end{align}
\item $\tau$ is a strength for this comonad in the sense that the following diagrams commute.
\begin{align}
\diagram
1_\C \times \Psi  * \ar[d]^\tau \ar[dr]^{\pi_{\Psi *}}
\\ \Psi  \ar[r]^{\Psi  !} & \Psi  *
\enddiagram \label{tau1}
\end{align}
\begin{align}
\diagram
\Psi  \ar[r]^\delta \ar[d]^{\epsilon_0 \times \Psi  !} & \Psi ^2 \ar[d]^{\Psi  \epsilon_0} \\
1_\C \times \Psi * \ar[r]^-\tau & \Psi  
\enddiagram & &
\diagram
1_\C \times \Psi  * \ar[d]^\tau \ar[r]^-{\tau \times \delta} & \Psi (1_\C \times \Psi  *) \ar[d]^{\Psi  \tau} \\
\Psi  \ar[r]^\delta & \Psi ^2
\enddiagram \label{tau2}
\end{align}
\end{enumerate}
\end{defn}

The word \emph{homotopical} is used to describe this functorial relation for the following reason. Suppose that we extract from the functorial relation $ R$ a notion of \emph{homotopy equivalence} in the usual way: where two objects $X$ and $Y$ are homotopic if there are morphisms $f: X \to Y $, $g: Y \to X$, $h: X \to \Psi X$, $i: Y \to \Psi Y$ such that $\epsilon_0 h = gf$, $\epsilon_1 h = 1_X$, $\epsilon_0 i = fg$, and $\epsilon_1 i = 1_Y$. Then the data given in the above definition provide a homotopy between every $X$ and $\Psi X$. 

\begin{exmp}\label{hoeqex4}
More generally, consider the relation in Example \ref{hoeqex3}.

Suppose that there is a morphism $d$ making the following diagrams commute.
\[ \diagram 
I \ar[r]^{1! \times I} \ar@{=}[dr] & I^2 \ar[d]^d & \ar[l]_{I \times 1!} \ar@{=}[dl] I \\
& I
\enddiagram \ \ \ \ \ 
 \diagram 
I \ar[r]^{0! \times I} \ar[dr]_{0!} & I^2 \ar[d]^d & \ar[l]_{I \times 0!} \ar[dl]^{0!} I \\
& I
\enddiagram 
\ \ \ \ \ 
 \diagram 
 I^3 \ar[r]^{I \times d} \ar[d]^{d \times I} & I^2 \ar[d]^d\\
 I^2 \ar[r]^d& I
\enddiagram
\]
Then taking $X^d: X^I \to (X^I)^I$ for $\delta_X$ and $X^!: X \to X^I$ for $\tau_X$ makes this relation strictly homotopical.

For example, in the category $\Cat$, there is such a $d$ when $I$ is $\M$ or the groupoid generated by the graph $I: 0 \to 1$. Let the following diagram denote the graph $(I: 0 \to 1)^2$.
\[ \diagram
00 \ar[r]^{0I} \ar[d]^{I0}& 01 \ar[d]^{I1}\\
10 \ar[r]^{1I}& 11
\enddiagram \]
Then in either case, $d$ is generated by sending $0I$ and $I0$ to the identity morphism on $0$, and $I1$ and $1I$ to $I: 0 \to 1$.
\end{exmp}

\begin{exmp}\label{moorepaths2}
Consider the functorial relation $G$ on topological spaces described in Example \ref{moorepaths1}. 

There is a natural transformation $\delta_X: \Gamma X \to \Gamma^2 X$ which takes a pair $(p,r)$ to the standard path from $c(p(0))$ to $(p,r)$. To be precise, it maps $(p,r)$
to $(q,r)$ where $q(t) = (p_t, t) \in \Gamma X$ and $p_t|_{[0,t]} = p|_{[0,t]}$  for each $t \in \R^+$. 

There is a natural transformation $\tau_X: X \times \Gamma(*) \to \Gamma X$. The space $\Gamma(*)$ is isomorphic to $\R^+$, so it maps a pair $(x, r) \in X \times \R^+$ to the constant path at $x$ of length $r$.

These natural transformations make $G$ into a strictly homotopical functorial relation.
\end{exmp}

In the following lemma, we record a natural transformation $\tilde \tau$ whose existence is equivalent to that of $\tau$, but which will make the proof of the following proposition clearer.

\begin{lem}
Consider a strictly homotopical functorial relation as above. For any $f: X \to Y$, let $\tilde \tau_f:  \pullb{X}{\Psi  Y}{f}{\epsilon_0}  \to \Psi  X$ be the composite
\[\pullb{X}{\Psi  Y}{f}{\epsilon_0} \xrightarrow{1 \times \Psi  !} X \times \Psi  * \xrightarrow{\tau} \Psi  X. \]
It makes the following diagrams commute.
\begin{align}
\diagram 
& \pullb{X}{\Psi  Y}{f}{\epsilon_0} \ar[dr]_{\pi} \ar[dl]^{\pi} \ar[d]^{\tilde \tau} \\
X& \Psi  X \ar[l]^{\epsilon_0} \ar[r]_{\epsilon_1} & X
\enddiagram & & 
\diagram
X \ar[r]^-{1 \times \eta f} \ar[dr]^\eta & \pullb{X}{\Psi  Y}{f}{\epsilon_0} \ar[d]^{\tilde \tau} \\
& \Psi  X
\enddiagram
\label{strength1}
\end{align}
\begin{align}
\diagram
\pullb{X}{\Psi  Y}{f}{\epsilon_0} \ar[d]^-{\tilde \tau} \ar[r]^-{\delta \pi_{\Psi  Y}} & \Psi ^2 Y \ar[d]^{\Psi  \epsilon_0} \\
\Psi  X \ar[r]^{\Psi  f} & \Psi  Y
\enddiagram & &
\diagram
\pullb{X}{\Psi  Y}{f}{\epsilon_0} \ar[d]^{\tilde \tau} \ar[r]^{\tilde \tau \times \delta} &  \Psi (\pullb{X}{\Psi  Y}{f}{\epsilon_0}) \ar[d]^{\Psi  \tilde \tau  } \\
\Psi  X \ar[r]^\delta & \Psi ^2 X
\enddiagram
\label{strength2}
\end{align}
\end{lem}

\begin{proof}
The commutativity of these diagrams is equivalent to that of the corresponding diagrams in \eqref{taulift}, \eqref{tau1}, and \eqref{tau2}.
\end{proof}

\begin{prop}\label{leftalg}
Let $ R$ be a strictly homotopical functorial relation on $\C$. Then the functor $\lambda: \C^\M \to \C^\M$ underlies a comonad on $\C^\M$ where the components of the counit and comultiplication at each object $f: X \to Y$ in $\C^\M$ are given by the following diagrams
$$
\diagram
X \ar[d]^{\lambda_f} \ar@{=}[r] & X \ar[d]^f\\
Mf  \ar[r]^-{\rho_f} & Y
\enddiagram  \ \ \
\diagram
X \ar[d]^{\lambda_f} \ar@{=}[r]  & X \ar[d]^{\lambda^2_f}  \\
Mf    \ar[r]^-{1 \times \tilde \tau \times \delta}   &  
M \lambda_f 
\enddiagram 
$$
where the morphism $1 \times \tilde \tau \times \delta$ is the composition
\[ \pullb{X}{\Psi Y}{f}{\epsilon_0}  \xrightarrow{1_X \times \tilde \tau_f \times \delta_Y } \pullb{X}{}{\lambda_f}{\epsilon_0}(\pullb{\Psi  X}{\Psi ^2 Y}{\Psi  f}{\Psi  \epsilon_0})  \cong \pullb{X}{}{\lambda_f}{\epsilon_0} \Psi  (\pullb{X}{\Psi  Y}{ f}{ \epsilon_0})  . \]
\end{prop}
\begin{proof}
We have already seen that the counit square commutes. To define $1_X \times \tilde \tau_f \times \delta_Y$ we make use of the commutativity of \eqref{deltalift} and the left hand sides of \eqref{strength1} and \eqref{strength2}. The commutativity of the comultiplication square above is given by the commutativity of \eqref{deltalift} and the right-hand diagram of \eqref{strength1}.

The following diagrams display the comonad axioms. The commutativity of 
$$ \diagram 
\pullb{X}{\Psi  Y}{f}{\epsilon_0}  
 & &  \pullb{X}{}{\lambda_f}{\epsilon_0} \Psi  (\pullb{X}{\Psi  Y}{ f}{ \epsilon_0})    \ar[ll]_-{1 \times \Psi  (\epsilon_1 \pi_{\Psi  Y})} \ar[rr]^-{ \epsilon_1 \pi_{\Psi (X \times \Psi  Y)}} 
 & & \pullb{X}{\Psi  Y}{f}{\epsilon_0}   \\
& & \pullb{X}{\Psi  Y}{f}{\epsilon_0}   \ar[u]|-{1 \times \tilde \tau \times \delta} \ar@{=}[ull] \ar@{=}[urr]
\enddiagram $$
follows from the commutativity of the left-hand diagrams in \eqref{comonad} and \eqref{strength1}, and the commutativity of
$$ \diagram
\pullb{X}{\Psi  Y}{f}{\epsilon_0}   \ar[d]^{1 \times \tilde \tau_f \times \delta_Y} \ar[rr]^-{1 \times \tilde \tau_f \times \delta_Y}
& & \pullb{X}{}{\lambda f}{\epsilon_0} \Psi  (\pullb{X}{\Psi  Y}{ f}{ \epsilon_0}) \ar[d]^{1 \times \tilde \tau_{\lambda } \times \delta_{X \times \Psi  Y}} 
\\ \pullb{X}{}{\lambda f }{\epsilon_0} \Psi  (\pullb{X}{\Psi  Y}{ f}{ \epsilon_0}) \ar[rr]^-{1 \times \Psi (1 \times \tilde \tau_f \times \delta_Y)} 
& & \pullb{X}{}{\lambda ^2f}{\epsilon_0} \Psi  (\pullb{X}{}{\lambda f }{\epsilon_0} \Psi  (\pullb{X}{\Psi  Y}{ f}{ \epsilon_0}))
\enddiagram
$$
follows from the right-hand diagrams in \eqref{comonad} and \eqref{strength2}.
\end{proof}

\subsubsection{Strictly symmetric functorial relations}

\begin{defn}
Say that a functorial relation $ R$ on $\C$ is \emph{strictly symmetric} if there exists a natural isomorphism $\iota$ with components
\[ \iota_X : \Psi  X \to \Psi  X \]
for each object $X$ of $\C$
such that $\iota \eta = \eta$, $\epsilon_0 \iota = \epsilon_1$, and $\epsilon_1 \iota = \epsilon_0$ in the diagram below.
\begin{align}\diagram
\Psi  \ar@<1.5ex>[d]^-{\epsilon_1} \ar@<-1.5ex>[d]_-{\epsilon_0} \ar[rr]^-{\iota} & &  \Psi  \ar@<1.5ex>[d]^-{\epsilon_0 } \ar@<-1.5ex>[d]_-{\epsilon_1 }
 \\  1_\C   \ar[u]|-{ \eta} \ar@{=}[rr] & & 1_\C   \ar[u]|-{\eta}
\enddiagram \label{iotalift}\end{align}
\end{defn}

If $ R$ is a monic relation, then the definition of \emph{strictly symmetric} given here coincides with the usual definition of \emph{symmetric}.

\begin{exmp}\label{hoeqex5}
More generally, consider the relation in Example \ref{hoeqex3}.

Suppose that there an isomorphism $i$ making the following diagrams commute.
\[ \diagram 
* \ar[d]^{n} \ar[dr]^{n-1}& \\
I \ar[r]^i & I
\enddiagram
\]
for $n \in \mathbb Z / 2$.
Then taking $X^i: X^I \to X^I$ for $\iota_X$ makes this relation strictly symmetric.

For example, in the category $\Cat$, there is such an $i$ when $I$ is the groupoid generated by the graph $I: 0 \to 1$.\end{exmp}

\begin{exmp}\label{moorepaths3}
Consider the functorial relation $G$ on topological spaces described in Example \ref{moorepaths1}. 

There is a natural transformation $\iota_X: \Gamma X \to \Gamma X$ which takes a pair $(p,r)$ to the pair $(q,r)$ where $q(t) = p(r-t)$ on $[0,r]$.

This makes $G$ into a strictly symmetric functorial relation.
\end{exmp}

\begin{lem}\label{standardfibration}
Consider a strictly symmetric, strictly transitive functorial relation $ R$ on $\C$. Then for every object $X$ of $\C$, the morphism 
\[\Psi X \xrightarrow{\epsilon_0 \times \epsilon_1} X \times X \]
has a $\reltofact  ( R)$-algebra structure.
\end{lem}
\begin{proof}
We need to show that there is a solution to the following lifting problem.
\[ \diagram 
\Psi  X \ar[d]_{\lambda(\epsilon_0 \times \epsilon_1)} \ar@{=}[r] & \Psi  X \ar[d]^{\epsilon_0 \times \epsilon_1} \\
\pullb{\Psi X}{\Psi (X \times X)}{(\epsilon_0 \times \epsilon_1)}{\epsilon_0} \ar[r]^-{\rho(\epsilon_0 \times \epsilon_1)} \ar@{-->}[ur] & X \times X
\enddiagram \]
We will do this by finding two lifts $a$ and $b$ as illustrated below.
\[
 \diagram 
\Psi  X \ar[dd]_{1 \times \eta (\epsilon_0 \times \epsilon_1)} \ar@{=}[rr] \ar[dr]_{\eta \epsilon_0 \times 1 \times \eta \epsilon_1 }&& \Psi  X \ar[dd]^{\epsilon_0 \times \epsilon_1}  \\ 
& \pullb{\Psi X}{\Psi X}{\epsilon_0}{\epsilon_0} \pullb{}{\Psi X}{\epsilon_1}{\epsilon_0} \ar@{-->}[ur]^b  \ar[dr]^{\epsilon_1 \pi_0 \times \epsilon_1 \pi_2}\\
\pullb{\Psi X}{\Psi (X \times X)}{(\epsilon_0 \times \epsilon_1)}{\epsilon_0} \ar[rr]_-{\epsilon_1 \pi_1} \ar@{-->}[ur]^a && X \times X
\enddiagram \tag{$*$} \]

Let $u: \Psi (X \times X) \to \Psi X \times \Psi X$ denote the universal morphism induced by the universal property of $\Psi X \times \Psi X$. It makes the following diagram commute.
\[ \diagram 
\Psi  X \ar[d]_{1 \times \eta (\epsilon_0 \times \epsilon_1)} \ar[r]^-{1 \times \eta \epsilon_0 \times \eta \epsilon_1} & \pullb{\Psi X}{(\Psi  X \times \Psi X)}{(\epsilon_0 \times \epsilon_1)}{(\epsilon_0 \times \epsilon_0)}  \ar[d]^{(\epsilon_1 \times \epsilon_1) \pi_{(\Psi X \times \Psi X)}} \\
\pullb{\Psi X}{\Psi (X \times X)}{(\epsilon_0 \times \epsilon_1)}{\epsilon_0} \ar[r]^-{\epsilon_1 \pi_{\Psi (X \times X)}} \ar[ur]^{1 \times u} & X \times X
\enddiagram \]
Note that the outside square of this diagram is isomorphic to the lower-left portion of diagram ($*$). Therefore, $1 \times u$ is the lift $a$ that we seek.

Now we let $b: \pullb{\Psi X}{\Psi X}{\epsilon_0}{\epsilon_0} \pullb{}{\Psi X}{\epsilon_1}{\epsilon_0}  \to \Psi X$ be the following composite.
\[\pullb{\Psi X}{\Psi X}{\epsilon_0}{\epsilon_0} \pullb{}{\Psi X}{\epsilon_1}{\epsilon_0}  \xrightarrow{1 \times \mu} 
\pullb{\Psi X}{\Psi X}{\epsilon_0}{\epsilon_0} \xrightarrow{ \iota \times 1} \pullb{\Psi X}{\Psi X}{\epsilon_1}{\epsilon_0}
 \xrightarrow{\mu} \Psi X. \]
This $b$ makes the upper right-hand portion of the above diagram commute.

Therefore, we have found a lift in the original diagram, and shown that $\epsilon_0 \times \epsilon_1$ has a $\reltofact  ( R)$-algebra structure.
\end{proof}

\begin{thm} \label{symthm}
Consider a strictly symmetric functorial relation $ R:  \C \to \C^\reldiag$ such that $ \reltofact  ( R)$ is a weak factorization structure and such that every morphism
\[\Psi X \xrightarrow{\epsilon_0 \times \epsilon_1} X \times X \]
is in $\alg{{\reltofact  ( R)}}$. Then the class $\coalg{{\reltofact  ( R)}}$ is stable under pullback along $\alg{{\reltofact  ( R)}}$.
\end{thm}
\begin{proof}
Consider the following pullback
\[ \diagram
A \times_Y X \ar[d]_{\pi_X} \ar[r] \pullback & A \ar[d]^\ell  \\
X \ar[r]^r & Y
\enddiagram \]
where $r$ is in $\alg{{\reltofact  ( R)}}$, and $\ell$ is in $\coalg{{\reltofact  ( R)}}$.

The morphism $\pi_X$ is in $\coalg{{\reltofact  ( R)}}$ if and only if there is a solution to the following lifting problem.
\[ \diagram
A \times_Y X \ar[d]_{\pi_X} \ar[r]^-{\lambda_{\pi_X}}  & \pullb{A}{\Psi X}{\ell}{r \epsilon_0}   \ar[d]^{\rho_{\pi_X}}  \\
X \ar@{=}[r] \ar@{-->}[ur]^s & X
\enddiagram  \tag{$*$}\]
We will construct such a lift.

Since $\ell$ is in $\coalg{{\reltofact  ( R)}}$, there is a lift $a$ in the following square.
\[ \diagram
A \ar[d]_{\ell} \ar[r]^-{\lambda_\ell}  & \pullb{A}{\Psi Y}{\ell}{\epsilon_0}   \ar[d]^{\rho_\ell}  \\
Y \ar@{=}[r] \ar@{-->}[ur]^a & Y
\enddiagram \]

Since $r$ is in $\alg{{\reltofact  ( R)}}$, the morphism $r \times 1_X: X \times X \to Y \times X$ is in $\alg{{\reltofact  ( R)}}$ (as it is a pullback of $r$), and then the morphism $r \epsilon_0 \times \epsilon_1 : \Psi X \to Y \times X$ is in $\alg{{\reltofact  ( R)}}$ (as it is the composition of $\epsilon_0 \times \epsilon_1 \in \alg{{\reltofact  ( R)}}$ and $r \times 1_X  \in \alg{{\reltofact  ( R)}}$). 
Thus, there is a lift in the following square.
\[ \diagram
X \ar[d]_{\lambda_r} \ar[r]^\eta  & \Psi X  \ar[d]^{r \epsilon_0 \times \epsilon_1 }  \\
\pullb{X}{\Psi Y}{r}{\epsilon_0} \ar[r]^-{ \epsilon_1 \times 1} \ar@{-->}[ur]^b & Y \times X
\enddiagram \]

Now let $s$ be the following composition.
\[ X \xrightarrow{ar \times 1_X} \pullb{A}{\Psi Y}{\ell}{\epsilon_0}\pullb{}{X}{\epsilon_1}{r} \xrightarrow{\pi_X \times \iota_Y \times \pi_A} \pullb{X}{\Psi Y}{r}{\epsilon_0}\pullb{}{A}{\epsilon_1}{\ell} \xrightarrow{ \pi_A \times b} \pullb{A}{\Psi X}{\ell}{r \epsilon_0} \]
This makes the diagram ($*$) commute.
\end{proof}

\begin{cor}\label{ttwfs}
Consider a strictly symmetric, strictly transitive relation $ R$ on $\C$ such that $\reltofact  ( R)$ is a weak factorization structure.
Then $\reltofact  ( R)$ is type-theoretic.
\end{cor}
\begin{proof}
By the previous two results, we know that $\coalg{{\reltofact  ( R)}}$ is stable under pullback along $\alg{{\reltofact  ( R)}}$. By Proposition \ref{relfact2fib}, every object is fibrant. Thus, $\reltofact  ( R)$ is type-theoretic.
\end{proof}

\subsubsection{Summary}
We now have the following theorem.

\begin{thm}\label{rel2awfs}
Consider a strictly transitive, strictly homotopical functorial relation $ R$ on $\C$. Then the functorial factorization ${\reltofact }({ R})$ is an algebraic weak factorization structure on $\C$.
\end{thm}
\begin{proof}
By Proposition \ref{leftalg}, $\lambda_{{\reltofact }({ R})}$ underlies a comonad, and by Proposition \ref{rightalg}, $\rho_{{\reltofact }({ R})}$ underlies a monad.
 \end{proof}

\begin{defn}
A \emph{strict Moore relation structure} on $\C$ is a functorial relation $ R$ which preserves pullbacks together with the structure described in the definitions of strictly transitive, strictly homotopical, and strictly symmetric. A \emph{strict Moore relation system} on a category $\C$ with finite limits is a functorial relation $ R$ which preserves pullbacks and which is strictly transitive, strictly homotopical, and strictly symmetric.
\end{defn}

Then we have the following theorem.

\begin{thm}\label{algmrs2ttwfs}
Consider a strict Moore relation system $ R$ on $\C$. Then the functorial factorization ${\reltofact }({ R})$ is a type-theoretic, algebraic weak factorization structure on $\C$.
\end{thm}
\begin{proof}
By the previous theorem, ${\reltofact }({ R})$ is an algebraic weak factorization structure on $\C$. By Proposition \ref{ttwfs}, it is type-theoretic.
\end{proof}

\begin{exmp}
Consider the relation $G$ on the category $\T$ of topological spaces from Examples \ref{moorepaths1}, \ref{moorepaths2}, and \ref{moorepaths3}.
This generates a type-theoretic, algebraic weak factorization structure on $\C$ whose factorization of a morphism $f: X \to Y$ is
\[ X \xrightarrow{1_X \times c f} X \times_Y \Gamma Y \xrightarrow{\pi_Y} Y,\]
whose left class consists of trivial Hurewicz cofibrations, and whose right class consists of Hurewicz fibrations. (This weak factorization system was first described in \cite{Str72} while this particular weak factorization structure was originally described in \cite{May75}.)
\end{exmp}

\subsection{Moore relation systems}
\label{sec:wmrs}

In this section, we describe the minimal structure that a relation $ R$ on $\C$ needs to have so that $\reltofact  ( R)$ is a type-theoretic weak factorization structure. The minimality will be justified by Corollary \ref{ttwfs2mrs}, and though we do not give any examples in this section, many can be obtained from that corollary.

In what follows, we define what it means for a relation to be \emph{transitive}, \emph{homotopical}, and \emph{symmetric}. Note that while the properties required of a transitive relation can be easily seen to be weaker than the properties required of a \emph{strictly} transitive relation, the definitions of homotopical and symmetric given below differ more significantly from their strict predecessors.

In what follows, we will let $\lambda$ denote $\lambda_{\reltofact  ( R)}$, $\rho$ denote $\rho_{\reltofact  ( R)}$, and $M$ denote $\cod \lambda = \dom \rho$.

\subsubsection{Transitive relations}

\begin{defn} \label{weaktrans}
Say that a relation $ R$ on $\C$ is \emph{transitive} if there exists a morphism \[\mu_X : \pullb{\Psi  X}{\Psi  X}{\epsilon_1}{\epsilon_0} \to \Psi  X\]
for every object $X$ of $\C$
such that the following diagrams commute.
\begin{align}\diagram
  \pullb{\Psi  X}{\Psi X }{\epsilon_1}{\epsilon_0} \ar@<0.5ex>[d]^-{\epsilon_1 \pi_1 } \ar@<-0.5ex>[d]_-{\epsilon_0 \pi_0 } \ar[rr]^-{\mu} & &  \Psi X \ar@<0.5ex>[d]^-{\epsilon_1 } \ar@<-0.5ex>[d]_-{\epsilon_0 }
 \\  X   \ar@{=}[rr] & & X 
\enddiagram 
& & &
\diagram
\Psi X \ar[r]^-{1 \times \eta}\ar@{=}[dr] & \pullb{\Psi X}{\Psi X}{\epsilon_1}{\epsilon_0} \ar[d]^{\mu} \\
& \Psi X
\enddiagram 
\label{weakmu}
\end{align}
\end{defn}

\begin{nexmp}\label{hononex}
Now we can see why the relation $G$ on the category $\T$ of topological spaces is more useful than the relation $H$ on $\T$ which sends every space $X$ to
\[ \diagram
 X   \ar[r]|-{X^!} &  \mathrm X^I \ar@<1.5ex>[l]^-{X^1 } \ar@<-1.5ex>[l]_-{X^0 }
\enddiagram \] 
as in Example \ref{hoeqex2a} where $I$ is the usual interval $[0,1]$.

Suppose that this relation is transitive with a $\mu: X^I {_{X^1} \times_{X^0} } X^I \to X^I$ of the form $X^m: X^{[0,2]} \to X^{[0,1]}$. Then $m$ would have to make the following diagrams commute for $i = 0,1$

\[ 
\diagram
* \ar[d]^i \ar[dr]^{i*2}   \\
I \ar[r]^-m &  [0,2]
\enddiagram
\ \ \ \ \ \ \ 
\diagram
I \ar[r]^-m   \ar@/_/@{=}[dr] & [0,2]    \ar[d]^{s}\\
& I
\enddiagram \] 
where $s$ is the surjection which maps $[0,1]$ onto $[0,1]$ identically and $[1,2]$ onto the point $\{1\}$.
These diagrams say that $m(0) = 0$,  $m(1) = 2$, and $sm = 1$. But there is no such continuous function.
\end{nexmp}

%
%
%

\begin{prop}\label{weakrightalg}
Consider a transitive relation $ R$ on $\C$ as above.
Then for every morphism $f$ of $\C$, the morphism $\rho_f$ has a $\reltofact( R)$-algebra structure given by
\[\diagram
 M\rho_f  \ar[r]^-{1 \times \mu} \ar[d]^{\rho^2_f}& Mf \ar[d]^{\rho_f}
\\
Y \ar@{=}[r] & Y
 \enddiagram\]
where $1 \times \mu:  M\rho_f  \to Mf $ is the morphism 
\[1_X \times \mu_Y:   \pullb{X}{\Psi  Y}{f}{\epsilon_0}  \pullb{}{\Psi  Y}{\epsilon_1}{\epsilon_0} \longrightarrow  \pullb{X}{\Psi  Y}{f}{\epsilon_0} . \]
\end{prop}

\begin{proof}
The commutativity of the square in the statement follows from the commutativity of the left-hand diagram of \eqref{weakmu}.

It remains to check that the composition of the point with the algebra structure, $(1 \times \mu) \circ \lambda \rho(f)$, is the identity. 
\[ \diagram
\pullb{X}{\Psi  Y}{f}{\epsilon_0} \ar@{=}[drr] \ar[rr]^-{1 \times 1 \times \eta \epsilon_1} &&  \pullb{X}{\Psi  Y}{f}{\epsilon_0}  \pullb{}{\Psi  Y}{\epsilon_1}{\epsilon_0}   \ar[d]^{1 \times \mu} \\
&& \pullb{X}{\Psi  Y}{f}{\epsilon_0}
\enddiagram \]
The commutativity of this diagram follows from that of the right-hand diagram in \eqref{weakmu}. 
\end{proof}

As for the strictly transitive relations of the last section, when a relation $ R$ is monic, our definition of transitivity and the usual definition coincide.
\subsubsection{Homotopical relations}

The definition of \emph{transitive} could immediately be seen to be a weaker version of the definition of strictly transitive. This is not the case for the definition of \emph{homotopical}.

\begin{defn}\label{homotopical}
Say that a relation $ R$ on $\C$ is \emph{homotopical} if for each object $X$ of $\C$, there exists an object $\Psi ^\square X$ of $\C$ with morphisms 
\[ \diagram 
X \ar[r]^-\eta & \Psi ^\square X \ar@<1ex>[r]^{\epsilon_0} \ar[r]|{\epsilon_1} \ar@<-1ex>[r]_{\zeta}& \Psi  X
\enddiagram \]
\[ \delta_X: \Psi X \to \Psi ^\square X, \]
and for every morphism $f: X \to Y$, a morphism 
\[ \tau_f: \pullb{X}{\Psi ^\square Y}{ \eta f }{\zeta} \to \Psi (\pullb{X}{\Psi Y}{f}{\epsilon_0}) \]
which make the following diagrams commute.
\begin{align}
\diagram
\Psi X & \ar[l]_{\epsilon_i} \Psi ^\square X \ar[r]^{\zeta} & \Psi X & & & 
{\Psi ^\square Y} \ar[r]^{\epsilon_i} \ar[d]^\zeta & \Psi X \ar[d]^{\epsilon_0} \\
& X \ar[u]^\eta \ar[ul]^\eta \ar[ur]_\eta & & & &
\Psi X \ar[r]^{\epsilon_i} & X
\enddiagram 
\\
\diagram
X \ar[r]^\eta \ar[dr]_\eta & \Psi X \ar[d]^\delta & &\Psi X \ar[r]^{\epsilon_0} \ar[d]^\delta & X \ar[d]^\eta & & \Psi X \ar@{=}[rd] \ar[d]^\delta \\
& \Psi ^\square X & & \Psi ^\square X \ar@<.5ex>[r]^{\epsilon_0} \ar@<-.5ex>[r]_{\zeta} & \Psi X & & \Psi ^\square X \ar[r]^{\epsilon_1} & \Psi X
\enddiagram \label{weakho}
\\
\diagram
X \ar[r]^-{1 \times \eta f} \ar[dr]_{\eta(1 \times \eta f)} &  \pullb{X}{\Psi ^\square Y}{ \eta f }{\zeta} \ar[d]^\tau 
&\pullb{X}{\Psi ^\square Y}{ \eta f }{\zeta} \ar[dr]^{1 \times \epsilon_i} \ar[d]^\tau  \\
& \Psi (\pullb{X}{\Psi Y}{f}{\epsilon_0}) 
& \Psi (\pullb{X}{\Psi Y}{f}{\epsilon_0})  \ar[r]^-{\epsilon_i} & \pullb{X}{\Psi Y}{f}{\epsilon_0} 
\enddiagram \label{weakho2}\end{align}
where $i $ ranges over $ 0,1$.
\end{defn}

\begin{exmp}
The object $\Psi ^\square X$ will often (as in Proposition \ref{idpres2homo}) be the middle object of the factorization of the morphism $\eta: \Psi X \to \Psi ^{\times 4} X$ where 
 $\Psi ^{\times 4}X$ is the limit of the diagram below on the left and $\eta: \Psi X \to \Psi ^{\times 4} X$ is induced by the cone below on the right
\[ 
\diagram
X & \Psi X \ar[l]_{\epsilon_0} \ar[r]^{\epsilon_1} & X\\
\Psi X  \ar[u]_{\epsilon_0} \ar[d]^{\epsilon_1}  & &  \Psi X \ar[d]^{\epsilon_1} \ar[u]_{\epsilon_0}  &   \\
X & \Psi X \ar[l]_{\epsilon_0} \ar[r]^{\epsilon_1} & X
\enddiagram \ \  \  \ \ \  \diagram
X & \Psi X \ar[l]_{\epsilon_0} \ar[r]^{\epsilon_1} & X\\
\Psi X  \ar[u]_{\epsilon_0} \ar[d]^{\epsilon_1}  & \Psi X \ar[u]_{\eta \epsilon_0} \ar[d]^{\eta \epsilon_1} \ar@{=}[l] \ar@{=}[r] &  \Psi X \ar[d]^{\epsilon_1} \ar[u]_{\epsilon_0}  &   \\
X & \Psi X \ar[l]_{\epsilon_0} \ar[r]^{\epsilon_1} & X
\enddiagram  \]

In the category of topological spaces, this might look like the following. We use the relation $H$ here, as described in Non-example \ref{hononex}, though we ultimately are interested in the relation $G$. This is because the description involving $H$ is much easier to write down but still provides intuition to think about $G$.

Let $\delta(I \times I)$ denote the boundary of the unit square $I \times I$. 
Let $S$ denote the mapping cylinder of the function $\delta(I \times I) \to I$ which maps $(x,y)$ to $x$.
That is, $S$ is the quotient of $I \times \delta(I \times I)$ obtained by identifying the point $(1,x,y)$ with the point $(1,x,y')$ for any $(x,y),(x,y')$ in $\delta(I \times I)$.

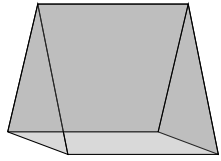
\[
\begin{tikzpicture}[fill opacity=.3,draw opacity=1]
\draw [fill = gray] (0,0.3) -- (.4,2) -- (.8,0) -- (0,0.3);
\draw [fill = gray] (2,0.3) -- (2.4,2) -- (2.8,0) -- (2,0.3);
\draw [fill = gray] (.8,0) -- (.4,2) -- (2.4,2) -- (2.8,0) -- (.8,0);
\draw [fill = gray] (0,0.3) -- (.4,2) -- (2.4,2) -- (2,0.3) -- (0,0.3);
\end{tikzpicture}
\]

Then let $I^{\square}X $ denote the space $X^S$ of all continuous functions from $S$ into $X$. The morphism $\eta: X \to I^\square X$ is the precomposition with the map $S \to *$. The projections $\epsilon_i, \zeta_i: I^\square X \to X^I$ are the precompositions of the inclusions of $I$ into each of the bottom edges in the illustration above.

There is a continuous function $S \to I$ which takes the bottom edges associated to $\epsilon_0$ and $\zeta_0$ and the top vertex above their intersection to the point $0 \in I$ and maps the top edge and the edges associated to $\epsilon_1$ and $\zeta_1$ each homeomorphically onto $I$. Precomposition with this continuous function is the morphism $\delta_X: X^I \to I^\square X$.

There is a homotopy equivalence $h: S \to I^2$ which commutes with the projections to $I^{\times 4}$. Then the composition \[ \pullb{X}{I^\square Y}{ \eta f }{\zeta_0} \xhookrightarrow{\eta \times h} \pullb{X^I}{ Y^{I \times I}}{  f^I }{\epsilon_0^I} \cong  (\pullb{X}{ Y^{I}}{  f }{\epsilon_0})^I \]
is the morphism $\tau_f$.

Now we can provide some intuition as to why we have switched from considering $\Psi ^2 X$ to $\Psi ^\square X$. In a space $\Gamma^2 X$, the lengths of the sides are coupled (e.g., for any $\gamma \in \Gamma^2 X$, $\Gamma \epsilon_0 \gamma $ has the same length as $\Gamma \epsilon_1 \gamma $) but this is not the case for $\Gamma^\square X$. In particular, the middle diagram of \ref{weakho} could not be satisfied if $\Gamma^\square X = \Gamma^{2}X$. To explain this from a slightly different perspective, when we obtain $\Psi ^\square X$ in this way, the morphism $\Psi ^\square X \to \Psi ^{\times 4}X$ is in the right class of the weak factorization system, giving it better behavior than $\Psi ^2 X \to \Psi ^{\times 4}X$.

This intuition will be given mathematical content when we extract this structure from any type-theoretic weak factorization structure in Proposition \ref{idpres2homo}.
\end{exmp}
\begin{prop}\label{weakleftalg} 
Let $ R$ be a homotopical relation on $\C$. Then for every morphism $f: X \to Y$ in $\C$, 
the morphism $\lambda_f$ has a $\reltofact( R)$-coalgebra structure given by
\[\diagram
 X  \ar@{=}[r] \ar[d]^{\lambda_f}& X \ar[d]^{\lambda^2_f} \\
Mf \ar[r]^{1 \times \tau \delta} & M \lambda_f
 \enddiagram\]
where $1 \times \tau \delta :  Mf  \to M\lambda_f $ is
\[1_X \times \tau_f \delta_Y: \pullb{X}{\Psi  Y}{f}{\epsilon_0} \to \pullb{X}{\Psi (\pullb{X}{\Psi  Y}{f}{\epsilon_0})}{1 \times \eta f}{\epsilon_0}  .\]
\end{prop}
\begin{proof}
The morphism $1_X \times \tau_f \delta_Y$ in the statement is induced from the morphisms $\pi_X : \pullb{X}{\Psi  Y}{f}{\epsilon_0} \to X$ and $\tau_f (1 \times \delta_Y): \pullb{X}{\Psi  Y}{f}{\epsilon_0} \to \Psi (\pullb{X}{\Psi  Y}{f}{\epsilon_0})$ by the universal property of the pullback $\pullb{X}{\Psi (\pullb{X}{\Psi  Y}{f}{\epsilon_0})}{1 \times \eta f}{\epsilon_0}$ because the following diagram commutes.
\[ \diagram
 &  & X \ar[dr]^{1 \times \eta f }\\
\pullb{X}{\Psi  Y}{f}{\epsilon_0} \ar[rrr]^{1 \times \eta \epsilon_0}\ar[urr]^{1_X} \ar[dr]^{1 \times \delta}& & &  \pullb{X}{\Psi  Y}{f}{\epsilon_0}  \\
& \pullb{X}{\Psi ^\square Y}{\eta f}{\zeta} \ar[urr]^{1 \times \epsilon_0} \ar[r]_{ \tau} & \Psi (\pullb{X}{\Psi  Y}{f}{\epsilon_0}) \ar[ur]_{\epsilon_0}
\enddiagram \]
The upper triangle commutes by the properties of the pullback in its domain. The lower left-hand triangle commutes because of the commutativity of the middle diagram in \eqref{weakho}. The lower right-hand triangle commutes because of the commutativity of the right-handle diagram in \eqref{weakho2}

The coalgebra square in the statement can be written more explicitly as
\[\diagram
 X  \ar@{=}[r] \ar[d]^{1 \times \eta f}& X \ar[d]^{1 \times \eta(1 \times \eta f)} \\
\pullb{X}{\Psi  Y}{f}{\epsilon_0} \ar[r]^-{1 \times \tau \delta} &\pullb{X}{\Psi (\pullb{X}{\Psi  Y}{f}{\epsilon_0})}{1 \times \eta f}{\epsilon_0}
 \enddiagram\]
 The commutativity of this square follows from the commutativity of the outside of the following diagram by the universal property of the pullback in the lower right-hand corner.
 \[\diagram
 X  \ar@{=}[r] \ar[d]^{1 \times \eta f}&  X  \ar@{=}[r]  \ar[d]^{1 \times \eta f} &  X \ar[d]^{1 \times \eta(1 \times \eta f)} \\
\pullb{X}{\Psi  Y}{f}{\epsilon_0} \ar[r]^-{1 \times  \delta} &\pullb{X}{\Psi ^\square Y}{\eta f}{\zeta} \ar[r]^-{1 \times \tau } &{X}\times{\Psi (\pullb{X}{\Psi  Y}{f}{\epsilon_0})}
 \enddiagram\]
 The left-hand square above commutes because the left-hand diagram of \eqref{weakho} commutes. The right-hand square commutes because the left-hand diagram of \eqref{weakho2} commutes.
 
Now it remains to check that the copoint composed with the coalgebra is the identity.
\[\diagram
 X  \ar@{=}[r] \ar[d]^{\lambda_f}& X \ar[d]^{\lambda^2_f}  \ar@{=}[r] & X \ar[d]^{\lambda_f} \\
Mf \ar[r]^{1 \times \tau \delta} \ar@{=}@/_2ex/[rr]& M \lambda_f \ar[r]^{\rho \lambda_f} & Mf 
 \enddiagram\]
 We have already seen that the two squares in this diagram commute. The composition $(\rho \lambda_f)(1 \times \tau \delta)$ is equal to the composition of the top and right sides of the diagram below.
 \[ \diagram 
\pullb{X}{\Psi  Y}{f}{\epsilon_0} \ar@{=}@/_/[drr] \ar[r]^-{1 \times  \delta} &\pullb{X}{\Psi ^\square Y}{\eta f}{\zeta_0} \ar[r]^-{1 \times \tau } \ar[dr]^{1 \times \epsilon_1} &{X}\times{\Psi (\pullb{X}{\Psi  Y}{f}{\epsilon_0})} \ar[d]^{\epsilon_1 \pi_1} \\
& & \pullb{X}{\Psi  Y}{f}{\epsilon_0}
 \enddiagram \]
 The commutativity of the left-hand triangle above follows from the commutativity of the right-hand diagram in \eqref{weakho}. The commutativity of the right-hand triangle above follows from the commutativity of the right-hand diagram in \eqref{weakho2}.
\end{proof}

\subsubsection{Symmetric relations}

\begin{defn}
Say that a relation $ R$ on $\C$ is \emph{symmetric} if there exist morphisms
\[ \nu_X : \Psi  X {_{\epsilon_0} \times_{\epsilon_0}} \Psi X \to \Psi  X \]
for every object $X$ of $\C$
such that the following diagrams commute.
\begin{align}\diagram
  \pullb{\Psi X }{\Psi X }{\epsilon_0}{\epsilon_0} \ar@<0.5ex>[d]^-{\epsilon_1 \pi_1 } \ar@<-0.5ex>[d]_-{\epsilon_1 \pi_0 } \ar[rr]^-{\nu} & &  \Psi X \ar@<0.5ex>[d]^-{\epsilon_1 } \ar@<-0.5ex>[d]_-{\epsilon_0 }
 \\  X   \ar@{=}[rr] & & X
\enddiagram 
& & &
\diagram
\Psi X \ar[r]^-{\eta \times 1}\ar@{=}[dr] & \pullb{\Psi X}{\Psi X}{\epsilon_0}{\epsilon_0} \ar[d]^{\nu} \\
& \Psi X
\enddiagram 
\end{align}
\end{defn}

This might look very different from the strict symmetry defined previously. But notice that if one takes $\iota_X: \Psi X \to \Psi X$ to be the following composite,
\[ \Psi X \xrightarrow{1 \times \eta \epsilon_0}   \pullb{\Psi X }{\Psi X }{\epsilon_0}{\epsilon_0}  \xrightarrow{\nu} \Psi X \]
then $\iota \eta = \eta$, $\epsilon_0 \iota = \epsilon_1$, and $\epsilon_1 \iota = \epsilon_0$ in the diagram below.
\[\diagram
\Psi  X\ar@<1.5ex>[d]^-{\epsilon_1} \ar@<-1.5ex>[d]_-{\epsilon_0} \ar[rr]^-{\iota} & &  \Psi X \ar@<1.5ex>[d]^-{\epsilon_0 } \ar@<-1.5ex>[d]_-{\epsilon_1 }
 \\  X  \ar[u]|-{ \eta} \ar@{=}[rr] & & X   \ar[u]|-{\eta}
\enddiagram \]
Thus, $\nu$ begets a more familiar symmetry, $\iota$.
However, we need the full strength of the morphism $\nu$ to prove the following lemma. 

\begin{lem}\label{standardfibrationweak}
Consider a symmetric, transitive relation $ R$ on $\C$. Then for every object $X$ of $\C$, the morphism 
\[\Psi X \xrightarrow{\epsilon_0 \times \epsilon_1} X \times X \]
has a $\reltofact( R)$-algebra structure.
\end{lem}

\begin{rem}
Note that the following proof for this Lemma is identical to that for the strict version (Lemma \ref{standardfibrationweak}) except that here we define $b$ to be $\nu (1 \times \mu)$ instead of $\mu(\iota \times 1)(1 \times \mu)$.
\end{rem}

\begin{proof}
We need to show that there is a solution to the following lifting problem.
\[ \diagram 
\Psi  X \ar[d]_{\lambda(\epsilon_0 \times \epsilon_1)} \ar@{=}[r] & \Psi  X \ar[d]^{\epsilon_0 \times \epsilon_1} \\
\pullb{\Psi X}{\Psi (X \times X)}{(\epsilon_0 \times \epsilon_1)}{\epsilon_0} \ar[r]^-{\rho(\epsilon_0 \times \epsilon_1)} \ar@{-->}[ur] & X \times X
\enddiagram \]
We will do this by finding two lifts $a$ and $b$ as illustrated below.
\[
 \diagram 
\Psi  X \ar[dd]_{1 \times \eta (\epsilon_0 \times \epsilon_1)} \ar@{=}[rr] \ar[dr]_{\eta \epsilon_0 \times 1 \times \eta \epsilon_1 }&& \Psi  X \ar[dd]^{\epsilon_0 \times \epsilon_1}  \\ 
& \pullb{\Psi X}{\Psi X}{\epsilon_0}{\epsilon_0} \pullb{}{\Psi X}{\epsilon_1}{\epsilon_0} \ar@{-->}[ur]^b  \ar[dr]^{\epsilon_1 \pi_0 \times \epsilon_1 \pi_2}\\
\pullb{\Psi X}{\Psi (X \times X)}{(\epsilon_0 \times \epsilon_1)}{\epsilon_0} \ar[rr]_-{\epsilon_1 \pi} \ar@{-->}[ur]^a && X \times X
\enddiagram \tag{$*$} \]

Let $u: \Psi (X \times X) \to \Psi X \times \Psi X$ denote the morphism induced by the universal property of $\Psi X \times \Psi X$. It makes the following diagram commute.
\[ \diagram 
\Psi  X \ar[d]_{1 \times \eta (\epsilon_0 \times \epsilon_1)} \ar[r]^-{1 \times \eta \epsilon_0 \times \eta \epsilon_1} & \pullb{\Psi X}{(\Psi  X \times \Psi X)}{(\epsilon_0 \times \epsilon_1)}{(\epsilon_0 \times \epsilon_0)}  \ar[d]^{(\epsilon_1 \times \epsilon_1) \pi_{(\Psi X \times \Psi X)}} \\
\pullb{\Psi X}{\Psi (X \times X)}{(\epsilon_0 \times \epsilon_1)}{\epsilon_0} \ar[r]^-{\epsilon_1 \pi_{\Psi (X \times X)}} \ar[ur]^{1 \times u} & X \times X
\enddiagram \]
Note that the outside square of this diagram is isomorphic to the lower-left triangle of diagram ($*$). Therefore, $1 \times u$ is the lift $a$ that we seek.

Now we let $b: \pullb{\Psi X}{\Psi X}{\epsilon_0}{\epsilon_0} \pullb{}{\Psi X}{\epsilon_1}{\epsilon_0}  \to \Psi X$ be the following composite.
\[\pullb{\Psi X}{\Psi X}{\epsilon_0}{\epsilon_0} \pullb{}{\Psi X}{\epsilon_1}{\epsilon_0}  \xrightarrow{1 \times \mu} 
\pullb{\Psi X}{\Psi X}{\epsilon_0}{\epsilon_0} \xrightarrow{ \nu}  \Psi X. \]
This $b$ makes the upper right-hand portion of the above diagram commute.

Therefore, we have found a lift in the original diagram and shown that $\epsilon_0 \times \epsilon_1$ has a $\reltofact( R)$-algebra structure.
\end{proof}

\begin{thm} \label{symthmweak}
Consider a symmetric relation $ R$ on $\C$ such that $\reltofact  ( R)$ is a weak factorization structure and such that every morphism
\[\Psi X \xrightarrow{\epsilon_0 \times \epsilon_1} X \times X \]
is in $\alg{\reltofact  ( R)}$. Then the class $\coalg{\reltofact  ( R)}$ is stable under pullback along $\alg{\reltofact  ( R)}$.
\end{thm}
\begin{proof}
The proof for this is identical to that for Theorem \ref{symthm}.
\end{proof}

\begin{cor}\label{ttwfsweak}
Consider a transitive, symmetric relation $ R$ on $\C$ such that the factorization $ \reltofact  ( R)$ is a weak factorization structure. Then $\reltofact  ( R)$ is type-theoretic.
\end{cor}
\begin{proof}
By the previous two results, we know that $\coalg{\reltofact  ( R)}$ is stable under pullback along $\alg{\reltofact  ( R)}$. By Proposition \ref{relfact2fib}, every object is fibrant. Thus, $\reltofact  ( R)$ is type-theoretic.
\end{proof}

\subsubsection{Summary}

Now we have the following theorem.

\begin{thm}\label{weakthm}
Consider a transitive and homotopical relation $ R$ on $\C$. Then $ {\reltofact }({  R})$ is a weak factorization structure.
\end{thm}
\begin{proof}
By Proposition \ref{weakleftalg}, every morphism in the image of $\lambda_{\reltofact  ( R)}$ has a $\reltofact  ( R)$-coalgebra structure, and by Proposition \ref{weakrightalg}, every morphism in the image of $\rho_{\reltofact  ( R)}$ has a $\reltofact  ( R)$-algebra structure.
\end{proof}

\begin{defn}
A \emph{Moore relation structure} on $\C$ is a relation $ R$ together with the structure given in the definitions of transitive, homotopical, and symmetric.
A \emph{Moore relation system} on $\C$ is a relation $ R$ together which is transitive, homotopical, and symmetric.
\end{defn}

Now we have the following theorem.

\begin{thm}\label{mrs2ttwfs}
Consider a Moore relation system $ R$ on $\C$. Then ${\reltofact }({  R})$ is a type-theoretic weak factorization structure.\end{thm}
\begin{proof}
By the previous theorem, $\reltofact ({  R})$ is a weak factorization structure.
Then by Corollary \ref{ttwfsweak}, $\reltofact ({  R})$ is type-theoretic.
\end{proof}

\section{$\Id$-presentations from type-theoretic weak factorization systems}
\label{sec:bigdiagram}

In this section, we consider a type-theoretic weak factorization structure $W$ on a finitely complete category $\C$. In the first section, \ref{sec:bigdiagramsubsec}, we show that the factorization $\reltofact \facttorel(W)$ is again a weak factorization structure equivalent to $W$. In the second section, \ref{sec:idpres}, we show that the relation $\facttorel(W)$ is an $\Id$-presentation of $[\reltofact \facttorel (W)] = [W]$. Combining these two results, we will have shown that any type-theoretic weak factorization system has an $\Id$-presentation.

\subsection{The main tool}\label{sec:bigdiagramsubsec}

Consider any type-theoretic weak factorization structure $W$ on $\C$.
Our aim in this section is to show that  $\reltofact \facttorel (W)$ is equivalent to $W$. However, we prove a slightly more general result which will become useful later (in Lemma \ref{rel2idpres}, Proposition \ref{idpres2homo}, and Proposition \ref{Idtypesonobj2mor}).

To that end, consider any relation $R$ 
with the following components at each object $X$ of $\C$
\[ \diagram
 X   \ar[r]|-{\eta_X} &   R X  \ar@<1.5ex>[l]^-{\epsilon_{1X} } \ar@<-1.5ex>[l]_-{\epsilon_{0X} }
\enddiagram \]
such that 
each $\eta_X: X \to R X$ is in $\coalg{W}$ and each $\epsilon_X = \epsilon_{0X} \times \epsilon_{1X}: R X \to X \times X$ is in $\alg{W}$. (We have in mind the relation $\facttorel(W)$ for our main result.)


Now we show that $ \reltofact  ( R) $ is a weak factorization structure equivalent to $W$. For readability, we will let $\lambda$ denote $\lambda_{\reltofact  ( R)}$ and $\rho$ denote $\rho_{\reltofact  ( R)}$. We need to show that (1) $\coalg{\reltofact  ( R)} =  \coalg{W}$, (2) $\alg{\reltofact  ( R)} = \alg{W}$, (3) $\lambda_f \in \coalg{W}$, and (4) $\rho_f \in \alg{W}$ for every morphism $f$ of $\C$. These facts are all relatively straightforward to show except (3) which appears as Proposition \ref{lambdacoalg}.

The hypothesis that $W$ is type-theoretic is integral to the proof below. In Lemma \ref{algright}, where we show fact (4), we need every object in $W$ to be fibrant. In Lemma \ref{coalgleft}, which is used to show fact (3) in Proposition \ref{lambdacoalg}, we need $\coalg{W}$ to be stable under pullback along $\alg{W}$.

\begin{lem} \label{algright}
For any morphism $f$ of $\C$, the morphism $\rho_f$ is in $\alg{W}$.
\end{lem}
\begin{proof}
Note first that $\pi_Y: X \times Y \to Y$ and $1_X \times  \epsilon_1  :\pullb{X}{R Y}{f}{\epsilon_0} \to X \times Y$ are in $\alg{W}$ because they are pullbacks of morphisms hypothesized to be in $\alg{W}$.
\begin{align*}
 \diagram
X \times Y \pullback \ar[r] \ar[d]^{\pi_Y} & X \ar[d]^{!} \\
Y \ar[r]^{!} & * 
\enddiagram & & \diagram
\pullb{X}{R Y}{f}{\epsilon_0}  \ar[d]^{1_X \times  \epsilon_1} \ar[r] \pullback & R  Y \ar[d]^{\epsilon_0 \times \epsilon_1} \\
X \times Y \ar[r]^{f \times 1_Y}& Y \times Y
\enddiagram 
\end{align*}
Since $\rho_f$ is the composition of these two maps, it is also in $\alg{W}$.
\end{proof}

\begin{lem} \label{coalgleft}
For any morphism $f$ in $\alg{W}$, the morphism $\lambda_f$ is in $\coalg{W}$,
\end{lem}
\begin{proof}
The morphism $\lambda_f$ is a pullback of $\eta \in \coalg{W}$ along $f \in \alg{W}$,
\[ \diagram
& \pullb{X}{R Y}{f}{\epsilon_0} \ar[dd] \ar[rr] \pullback & & R Y \ar[dd]^{\epsilon_0} \\
X \ar[ur]^-{\lambda_f} \ar[rr] \ar@{=}[dr] {\ar@{}[drr]|<<<<<<<{\text{\pigpenfont A}}} &  & Y \ar[ur]^{\eta} \ar@{=}[dr] &  \\
& X  \ar[rr]^f  & & Y 
\enddiagram \]
and since $W$ is type-theoretic, $\coalg{W}$ is stable under pullback along $\alg{W}$.
\end{proof}

\begin{prop}\label{samelifting}
We have that $\coalg{W} = \coalg{\reltofact  ( R)}$ and $\alg{W} = \alg{\reltofact  ( R)}$. 
\end{prop}

\begin{proof}
Consider a morphism $f$ in $\alg{\reltofact  ( R)}$. It is a retract of $\rho_f $. By Lemma \ref{algright}, $\rho_f $ is in $\alg{W}$. Since $\alg{W}$ is closed under retracts \cite[Prop.~14.1.8]{MP12}, $f$ is in $\alg{W}$.

Now consider a morphism $f$ in $\alg{W}$. Since $\lambda_ f$ is in $\coalg{W}$ by Lemma \ref{coalgleft}, $\lambda_f$ has the left lifting property against $f$. Therefore, $f$ is in $\alg{\reltofact  ( R)}$.

Thus, $\alg{W} = \alg{\reltofact  ( R)}$.

Now consider $\ell \in \coalg{W}$. Since $\ell$ has the left lifting property against $\alg{W}$, it has the left lifting property against $\rho_\ell$ in particular (Lemma \ref{algright}). Thus it is in $\coalg{\reltofact  ( R)}$.

Now suppose that $\ell \in \coalg{\reltofact  ( R)}$. Then for any $r \in \alg{W} = \alg{\reltofact  ( R)}$, $\ell$ has the left-lifting property against $r$ (Proposition \ref{coalgalglift}). Thus, $\ell$ is in $^\boxslash( \alg{W} )= \coalg{W}$.

Therefore, $\coalg{W} = \coalg{\reltofact  ( R)}$.
\end{proof}

\begin{prop}\label{lambdacoalg}
For any morphism $f$ of $\C$, the morphism $\lambda(f)$ is in $\coalg{W}$.
\end{prop}

\begin{proof}

We need to show that $\lambda_f$ has a $\lambda$-coalgebra structure, or that, equivalently, there is a solution to the following lifting problem.

\[ \diagram
X \ar[d]^{\lambda_f} \ar[r]^-{\lambda_{\lambda_f}}& \pullb{X}{R ( \pullb{X}{R Y}{f}{\epsilon_0})}{\lambda f}{\epsilon_0} \ar[d]^{\rho_{ \lambda_ f}} \\
\pullb{X}{R Y}{f}{\epsilon_0} \ar@{-->}[ur] \ar@{=}[r]& \pullb{X}{R Y}{f}{\epsilon_0} 
\enddiagram \]

First we define a new morphism $\mu: \pullb{R Y}{R Y}{\epsilon_1}{\epsilon_0} \to R Y$. Note that $\eta \epsilon_0 \times 1: R Y \to  \pullb{R Y}{R Y}{\epsilon_1}{\epsilon_0}$ is in $\coalg{W}$ since it is a pullback of a morphism in $\coalg{W}$ along a morphism in $\alg{W}$, as shown below.

\[ \diagram
& \pullb{R Y}{R Y}{\epsilon_1}{\epsilon_0} \ar[dd] \ar[rr] \pullback & & R Y \ar[dd]^{\epsilon_1} \\
R Y \ar[ur]^-{\eta \epsilon_0 \times 1} \ar[rr] \ar@{=}[dr] {\ar@{}[drr]|<<<<<<<{\text{\pigpenfont A}}} &  & Y \ar[ur]^{\eta} \ar@{=}[dr] &  \\
& R Y  \ar[rr]^{\epsilon_0} & & Y 
\enddiagram \]
Then, we define $\mu$ to be a solution to the following lifting problem.
\[ \diagram
R Y \ar@{=}[rr] \ar[d]^{\eta \epsilon_0 \times 1} & &  R Y \ar[d]^{\epsilon_0 \times \epsilon_1}  \\
 \pullb{R Y}{R Y}{\epsilon_1}{\epsilon_0}  \ar@{-->}[urr]^{\mu} \ar[rr]^-{\epsilon_0 \pi_0 \times \epsilon_1 \pi_1}&& Y \times Y
\enddiagram \]



\begin{sidewaysfigure}
\[ \diagram
X \ar[ddd]|\hole_{\lambda_f}|{=}^{1 \times \eta f} \ar[rr]_-{1 \times \eta f} \ar@/^4ex/[rrrrrrrr]^{\eta \lambda_f = \eta(1_X \times \eta f)}
&& \pullb{X}{R Y}{f}{\epsilon_0} \ar[ddd]|\hole_{\lambda_{ \rho_ f} }|{=}^{ 1_{X \times R Y} \times \eta \epsilon_1}  \ar[rrrr]_-{ \eta(1_{X \times R Y} \times \eta \epsilon_1) } 
&&&& R (\pullb{X}{R Y}{f}{\epsilon_0} \pullb{}{R Y}{\epsilon_1}{\epsilon_0}) \ar[ddd]^{ \epsilon_0 \times \epsilon_1 } \ar[rr]_-{R (1_X \times \mu)}
&&  R ( \pullb{X}{R Y}{f}{\epsilon_0}) \ar[ddd]^{\epsilon_0 \times \epsilon_1}
\\
\\
\\
\pullb{X}{R Y}{f}{\epsilon_0} \ar[rr]^-{1_X \times \eta f \times 1_{R Y}} \ar@/_4ex/[rrrrrrrr]_{(1_X \times \eta f) \times (1_{X \times R Y})}
&& \pullb{X}{R Y}{f}{\epsilon_0} \pullb{}{R Y}{\epsilon_1}{\epsilon_0} \ar[rrrr]^-{(1_{X \times R Y} \times \eta \epsilon_1) \times (1_{X \times R Y \times R Y})} \ar@{-->}[uuurrrr]^{\sigma}
&&&& (\pullb{X}{R Y}{f}{\epsilon_0} \pullb{}{R Y}{\epsilon_1}{\epsilon_0})^2 \ar[rr]^-{(1_X \times \mu)^2}
&& (\pullb{X}{R Y}{f}{\epsilon_0} )^2
\enddiagram \]
\caption{Lifting diagram}\label{figure:big_diagram} 
\end{sidewaysfigure}

Now we refer to figure Figure~\ref{figure:big_diagram} on page~\pageref{figure:big_diagram}. Since $\rho_f$ is in $\alg{W}$, we know that $\lambda_{\rho_f}$ is in $\coalg{W}$. Therefore, there is a lift $\sigma$ as illustrated in the figure.

Let $\sigma':  \pullb{X}{R Y}{f}{\epsilon_0}  \to R(\pullb{X}{R Y}{f}{\epsilon_0} )$ be the composite $R(1_X \times \mu) \sigma (1_X \times \eta f \times 1_{R Y})$ -- that is, the composite from the bottom left to top right of the diagram in Figure~\ref{figure:big_diagram}. Then a rearrangement of Figure~\ref{figure:big_diagram} produces the commutative diagram below, and $1_X \times \sigma'$ is our desired lift.
\[\diagram
X \ar[rr]^-{\lambda_{\lambda_ f} =}_-{ 1_X \times \eta(1_X \times \eta f)}
\ar[dd]|\hole_{\lambda_f}|{=}^{1 \times \eta f}
&& \pullb{X}{R ( \pullb{X}{R Y}{f}{\epsilon_0})}{ \lambda_f}{\epsilon_0} \ar[dd]|\hole_{\rho_{\lambda_ f}}|=^{\epsilon_1 \pi_{R(X \times R Y)}}
\\ \\
\pullb{X}{R Y}{f}{\epsilon_0} \ar@{=}[rr] \ar[uurr]^{1_X \times \sigma'}
&& \pullb{X}{R Y}{f}{\epsilon_0}
\enddiagram \]

Therefore, $\lambda_f$ is in $\coalg{W}$. \qedhere
\end{proof}

We put the preceding results together into the following theorems.

\begin{thm}\label{wfs2wfs}
Consider a type-theoretic weak factorization structure $W$ on $\C$. Consider a relation $R$ on $\C$ which has components 
\[ \diagram
 X   \ar[r]|-{\eta_X} &   R X , \ar@<1.5ex>[l]^-{\epsilon_{1X} } \ar@<-1.5ex>[l]_-{\epsilon_{0X} }
\enddiagram \]
such that $\eta_X$ is $\coalg{W}$ and $\epsilon_{0X} \times \epsilon_{1X} : R(X) \to X \times X$ is in $\alg{W}$ at each object $X$ of $\C$.
Then the factorization $ \reltofact  (R) $ is a weak factorization structure equivalent to $W$. 
\end{thm}
\begin{proof}
By Lemma \ref{algright} and Proposition \ref{samelifting}, every morphism in the image of $\rho_{ \reltofact  (R)}$ is in $\alg{ \reltofact  (R)}$. By Proposition \ref{lambdacoalg} and Proposition \ref{samelifting}, every morphism in the image of $\lambda_{ \reltofact  (R)}$ is in $\coalg{ \reltofact  (R)}$.
 Thus, $ \reltofact  (R)$ is a weak factorization structure. By Propositions \ref{equivwfs} and \ref{samelifting}, it is equivalent to $W$.
\end{proof}

The following corollary is the main result of this section.

\begin{cor}\label{wfs2wfscor}
Consider a type-theoretic weak factorization structure $W$ on $\C$. The factorization $ \reltofact \facttorel(W)$ is a weak factorization structure equivalent to $W$. 
\end{cor}
\begin{proof}
We need to show that the relation $ \facttorel(W)$ can be substituted for $R $ in the statement of the previous theorem, \ref{wfs2wfs}. In the notation of the previous theorem, \ref{wfs2wfs}, $\eta_X$ is $\lambda_{W}({\Delta_X})$ and $\epsilon_X$ is $\rho_W({\Delta_X})$, and these are in the left and right class, respectively, as required.
\end{proof}

The following corollary will become a useful technical device (in Proposition \ref{idpres2homo}) and is the reason that we proved Theorem \ref{wfs2wfs} in more generality than needed for Corollary \ref{wfs2wfscor}.

\begin{cor}\label{newfact}
Consider a type-theoretic weak factorization structure $W$ on $\C$. Consider a relation on just one object $Y$ of $\C$ with the following components
\[ \diagram
 Y   \ar[r]|-{\eta_Y} &   R Y , \ar@<1.5ex>[l]^-{\epsilon_{1Y} } \ar@<-1.5ex>[l]_-{\epsilon_{0Y} }
\enddiagram \]
such that $\eta_Y$ is in the left class and $\epsilon_{0Y} \times \epsilon_{1Y}: RY \to Y \times Y$ is in the right class of $W$. Then for any morphism $f: X \to Y$ of $\C$, in the following factorization
\[ \diagram
X \ar[r]^-{1 \times \eta f}& X \times_{\epsilon_0 } RY \ar[r]^-{\epsilon_1 \pi_{RY}}& Y
\enddiagram \]
the morphism $1 \times \eta f$ is in the left class, and $\epsilon_1 \pi_{RY}$ is in the right class of $W$.
\end{cor}
\begin{proof}
Consider the relation $ \facttorel(W)$. We construct a new relation $S$ which coincides with $ \facttorel(W)$ everywhere except at $Y$. So set $S(X) =  \facttorel(W)(X)$ for every object $X \neq Y$ and set $S(Y) = R$. Then a lift of any morphism with domain or codomain $Y$ can be extracted from the weak factorization structure $W$. That is, a lift of any morphism $f: X \to Y$ can be obtained as a solution to the following lifting problem.
\[ \diagram
X \ar[r]^{\eta f} \ar[d]^{\eta}& \Id(Y) \ar[d]^{\epsilon_0 \times \epsilon_1}\\
\Id(X) \ar[r]^{f\epsilon_0 \times f\epsilon_1} \ar@{-->}[ur]&Y \times Y
\enddiagram \]
A lift of any morphism $g: Y \to Z$ can be obtained analogously.

The relation $S$ satisfies the hypotheses of Theorem \ref{wfs2wfs} so $\reltofact ( S) $ is a weak factorization structure equivalent to $W$. But $\reltofact ( S)  $ sends a morphism $f: X \to Y$ to the factorization in the statement. Thus $1 \times \eta f$ is in the left class and $\epsilon_1 \pi_{RY}$ is in the right class of $W$.
\end{proof}

\begin{rem}
Some might be opposed to the reference of equality of objects. However, it is not strictly necessary. One can emulate the proof of Theorem \ref{wfs2wfs}, replacing $\lambda_f$ with the $1 \times \eta f$ of the statement of Corollary \ref{newfact} and replacing the other occurences of $\lambda$ and $\rho$ with $\lambda_{\reltofact \facttorel(W)}$ and $\rho_{\reltofact \facttorel(W)}$ (essentially, making the replacement of $\facttorel(W)(Y)$ by $R(Y)$, as is done in the proof of Corollary \ref{newfact}, just where necessary in the proof of Theorem \ref{wfs2wfs}). Then we will obtain Corollary \ref{newfact}.
\end{rem}

\subsection{$\Id$-presentations}
\label{sec:idpres} Now we can show that every type-theoretic weak factorization system has an $\Id$-presentation.

\begin{lem}\label{rel2idpres}
Consider a relation $R$ on $\C$ such that $\reltofact ( R)$ is a type-theoretic weak factorization structure.
Denote the components of $R(X)$ for any object $X$ of $\C$ by the following diagram.
\[ \diagram
 X   \ar[r]|-{\eta_X} &  R X  \ar@<1.5ex>[l]^-{\epsilon_{1X} } \ar@<-1.5ex>[l]_-{\epsilon_{0X} }
\enddiagram \]

 Then $R$ is an $\Id$-presentation of the weak factorization system $[\reltofact ( R)]$ if and only if $\epsilon_{0X} \times \epsilon_{1X}: RX \to X \times X$ is in the right class for each object $X$.
\end{lem}
\begin{proof}
Suppose that $R$ is an $\Id$-presentation. Then by definition, we must have that each $\epsilon_{0X} \times \epsilon_{1X}: RX \to X \times X$ is in $\alg{\reltofact(R)}$.

Conversely, suppose that each $\epsilon_{0X} \times \epsilon_{1X}: RX \to X \times X$ is in $\alg{\reltofact(R)}$. Then it remains to show that each $f^*\eta_Y$, as displayed in the diagram ($*$) below, is in $\coalg{\reltofact(R)}$.
\[ \diagram 
&f^*RY\ar[rrr] \ar[dd]|!{"2,1";"2,4"}\hole &&&{RY} \ar[dd]^{\epsilon_{iY}} \\
X \ar[rrr] \ar@{=}[dr] \ar[ur]^{f^* \eta_Y} &&& Y \ar@{=}[dr] \ar[ur]^{\eta_Y} & \\
& X \ar[rrr]^f && & Y
\enddiagram \tag{$*$}\]

Note that when $i=0$ in the diagram $(*)$ above, the morphism $f^* \eta_Y$ is isomorphic to $\lambda_{\reltofact(R)}(f)$ (i.e., it has the same universal property as $1_X \times \eta_Y f: X \to X \times_Y RY$). Thus, it must be in $\coalg{\reltofact(R)}$.

There is an involution $I$ on $\rel[\C]$ which sends $S(X) \epsilon_i$ to $S(X) \epsilon_{i+1}$ for any $S \in \rel[\C]$, $X \in \C$, and $i \in \mathbb Z / 2$ and keeps all else constant. Then $IR$ satisfies the hypotheses of Theorem \ref{wfs2wfs}, so $\reltofact(IR)$ is a weak factorization structure equivalent to $\reltofact(R)$. Now when $i=1$, the morphism $f^* \eta_Y$ in the diagram $(*)$ is isomorphic to $\lambda_{\reltofact(IR)} f$, so it is in $\coalg{\reltofact(R)}$.

Therefore, every $f^*\eta_Y$ in the diagram $(*)$ is in $\coalg{\reltofact(R)}$, so $R$ is an $\Id$-presentation of this weak factorization system.
\end{proof}

Now combining Corollary \ref{wfs2wfscor} with this lemma, \ref{rel2idpres}, we see the following.
\begin{thm}\label{wfs2idpres}
Consider a type-theoretic weak factorization structure $W$ on $\C$. The relation $ \facttorel(W)$ is an $\Id$-presentation of the weak factorization system $[W]$. Thus, every type-theoretic weak factorization system has an $\Id$-presentation.
\end{thm}
\begin{proof}
By Corollary \ref{wfs2wfscor}, $ \facttorel(W)$ generates a type-theoretic weak factorization structure $\reltofact \facttorel (W)$ equivalent to $W$. For any object $X$ of $\C$, the morphism 
\[ \facttorel(W) X\epsilon_0 \times  \facttorel(W) X \epsilon_1:  \facttorel(W) X \Phi \to X \times X\]
is the right factor of the morphism $\Delta(X)$ in the factorization $W$. Thus, it is in the right class of the weak factorization system.
Then this is an $\Id$-presentation of $[W]$ by Lemma \ref{rel2idpres}.
\end{proof}

\begin{cor}\label{wfs2idpresfunctors}
The functor $ \facttorel: \fact[\C] \to \rel[\C]  $ restricts to a functor $ \facttorel: \ttWFS[\C]\to \IdPres[\C]$, and the composition $|\reltofact \facttorel |: |\ttWFS[\C]| \to | \ttWFS[\C]|$ is isomorphic to the identity functor.
\end{cor}
\begin{proof}
By the previous theorem, \ref{wfs2idpres}, all objects in the image of $ \facttorel:  \ttWFS[\C] \to \rel[\C]  $ are $\Id$-presentations. Thus, this functor restricts to $ \facttorel: \ttWFS[\C]\to \IdPres[\C]$.

By the previous theorem again, for any object $W \in \ttWFS[\C]$, we have that $W$ is equivalent to $\reltofact \facttorel (W)$. Thus, they are isomorphic as objects of $|\ttWFS[\C]| $. Since $|\ttWFS[\C]| $ is a proset, these isomorphisms assemble into a natural transformation $1 \cong |\reltofact \facttorel |$.
\end{proof}

\begin{exmp}\label{cisex}
Given any Cisinski model structure $(\C, \W , \F)$ on a topos $\M$ \cite{Cis06}, we claim that the weak factorization system $(\C \cap \W \cap \M_{\F} , \F \cap \M_{\F})$ restricted to the full subcategory $\M_{\F}$ of fibrant objects is type-theoretic.

For this weak factorization system to be type-theoretic, all its objects must be fibrant, which we have satisfied by construction, and $\C \cap \W  \cap \M_{\F}$ must be stable under pullback along $\F  \cap \M_{\F}$. In a Cisinski model structure, $\C$ is precisely the class of monomorphisms, so it is stable under pullback (along all morphisms) in $\M$. Then, in particular, it is stable under pullback along $\F  \cap \M_{\F}$ in $\M_{\F}$. A standard result of model category theory says that $ \W \cap \M_{\F}$ is stable under pullback in  $\M_{\F}$ \cite[\S1~Ex.~1,\S4~Lem.~1]{Bro73}. Thus $\C \cap \W  \cap \M_{\F}$ is stable under pullback along $\F  \cap \M_{\F}$, and the weak factorization system $(\C \cap \W \cap \M_{\F} , \F \cap \M_{\F})$ is type-theoretic.

Then we find many examples of type-theoretic weak factorization systems, including those in the categories of Kan complexes \cite{Qui67}, quasicategories \cite{Joy08}, and fibrant cubical sets \cite{Cis06}. These all have $\Id$-presentations.
\end{exmp}

\section{Moore relations from type-theoretic weak factorization systems}
\label{sec:idpresandmrs}

In this section, we tie up the preceding sections by showing that a relation $R$ on a finitely complete category $\C$ is a Moore relation system if and only if it is an $\Id$-presentation.

We can immediately see from our previous results that any relation $R$ which underlies a Moore relation system is an $\Id$-presentation of the weak factorization system it generates.

\begin{prop}\label{mrs2idpres}
Consider a Moore relation system $R$ on $\C$. Then $R$ is an $\Id$-presentation.
\end{prop}
\begin{proof}
By Theorem \ref{mrs2ttwfs}, $R$ generates a type-theoretic weak factorization structure $ \reltofact  (R)$. By Lemma \ref{standardfibrationweak}, every $RX \epsilon_0 \times RX \epsilon_1$ is in the right class. Then by Lemma \ref{rel2idpres}, this is a $\Id$-presentation of $[\reltofact  (R)]$.
\end{proof}

Now we prove the converse: that any $\Id$-presentation is a Moore relation system.
We consider a relation $R$ on $\C$ which at an object $X$ gives the following diagram.
\[ \diagram
 X   \ar[r]|-{\eta_X} &  \Psi X , \ar@<1.5ex>[l]^-{\epsilon_{1X} } \ar@<-1.5ex>[l]_-{\epsilon_{0X} }
\enddiagram \]
We let $\lambda$ denote $\lambda_{\reltofact  (R)}$ and $\rho$ denote $\rho_{\reltofact  (R)}$.


\begin{prop}\label{idpres2trans}
Suppose that $R$ is an $\Id$-presentation of a weak factorization system. Then $R$ is transitive.
\end{prop}
\begin{proof}
For any object $X$ of $\C$, we let $\mu_X$ be a solution to the following lifting problem.
\[ \diagram
\Psi X \ar[d]_{\lambda_{\epsilon_1}} \ar@{=}[rr] & & \Psi X  \ar[d]^{\epsilon_0 \times \epsilon_1} \\
\pullb{\Psi X}{\Psi X}{\epsilon_1}{\epsilon_0} \ar[rr]_-{ \epsilon_0 \pi_0 \times \epsilon_1 \pi_1} \ar@{-->}[urr]^{\mu_X}& & X \times X
\enddiagram \]
This makes $R$ into a transitive relation.
\end{proof}

\begin{prop}\label{idpres2sym}
Suppose that $R$ is an $\Id$-presentation of a weak factorization system. Then $R$ is symmetric.
\end{prop}
\begin{proof}
For any object $X$ of $\C$, we let $\nu_X$ be a solution to the following lifting problem (where $\tau: \pullb{\Psi X}{\Psi X}{\epsilon_0}{\epsilon_0} \to \pullb{\Psi X}{\Psi X}{\epsilon_0}{\epsilon_0}$ is the standard twist involution).
\[ \diagram
\Psi X \ar[d]_{\tau \lambda_{\epsilon_0}} \ar@{=}[rr] & & \Psi X  \ar[d]^{\epsilon_0 \times \epsilon_1} \\
\pullb{\Psi X}{\Psi X}{\epsilon_0}{\epsilon_0} \ar[rr]_-{\epsilon_1 \pi_0  \times \epsilon_1 \pi_1} \ar@{-->}[urr]^{\nu_X}& & X \times X
\enddiagram \]
This makes $R$ into a symmetric relation.
\end{proof}

\begin{thm}\label{idpres2ttwfs}
Consider an $\Id$-presentation $R$ on $\C$. The factorization $ \reltofact  (R)$ is a type-theoretic weak factorization structure.
\end{thm}
\begin{proof}
By Proposition \ref{idpres2sym}, $R$ is symmetric. Then by Proposition \ref{relfact2fib} and Theorem \ref{symthmweak}, $ \reltofact  (R)$ is type-theoretic.
\end{proof}

\begin{cor}\label{idpres2ttwfsfunctors}
The functor $ \reltofact : \rel[\C] \to \fact[\C] $ restricts to a functor $ \reltofact : \IdPres[\C] \to \ttWFS[\C]$.
\end{cor}
\begin{proof}
The previous theorem, \ref{idpres2ttwfs}, tells us that every object in the image of $ \reltofact : \rel[\C] \to \fact[\C] $ is in the full subcategory $\ttWFS[\C]$. Thus, this functor restricts to $\reltofact : \IdPres[\C] \to \ttWFS[\C]$.
\end{proof}

\begin{prop}\label{idpres2homo}
Suppose that $R$ is an $\Id$-presentation of a weak factorization system. Then $R$ is homotopical.
\end{prop}
\begin{proof}
Let $\Psi ^{\times 4}X$ denote the limit of the following diagram in $\C$. 

\[ 
\diagram
X & \Psi X \ar[l]_{\epsilon_0} \ar[r]^{\epsilon_1} & X\\
\Psi X  \ar[u]_{\epsilon_0} \ar[d]^{\epsilon_1}  & &  \Psi X \ar[d]^{\epsilon_1} \ar[u]_{\epsilon_0}  &   \\
X & \Psi X \ar[l]_{\epsilon_0} \ar[r]^{\epsilon_1} & X
\enddiagram  \]
There is a morphism $u: \Psi  X \to \Psi ^{\times 4}X$ which is induced by the following cone.
\[ \diagram
X & \Psi X \ar[l]_{\epsilon_0} \ar[r]^{\epsilon_1} & X\\
\Psi X  \ar[u]_{\epsilon_0} \ar[d]^{\epsilon_1}  & \Psi X \ar@{=}[l] \ar@{=}[r] \ar[u]_{\eta \epsilon_0} \ar[d]^{\eta \epsilon_1} &  \Psi X \ar[d]^{\epsilon_1} \ar[u]_{\epsilon_0}  &   \\
X & \Psi X \ar[l]_{\epsilon_0} \ar[r]^{\epsilon_1} & X
\enddiagram  \]
Now we factor $u: \Psi X \to \Psi ^{\times 4}X$.
\[ \Psi X \xrightarrow{\lambda_{u}} Mu \xrightarrow{\rho_u} \Psi ^{\times 4} X  \]
Let $\Psi ^\square X$ denote $Mu$,
and let the following diagram denote the cone corresponding to $\rho_u  $. 
\[ \diagram
X & \Psi X \ar[l]_{\epsilon_0} \ar[r]^{\epsilon_1} & X\\
\Psi X  \ar[u]_{\epsilon_0} \ar[d]^{\epsilon_1}  & \Psi ^\square X \ar[u]_{\epsilon_0} \ar[d]^{\epsilon_1} \ar[l]_{\zeta_0} \ar[r]^{\zeta_1} &  \Psi X \ar[d]^{\epsilon_1} \ar[u]_{\epsilon_0}  &   \\
X & \Psi X \ar[l]_{\epsilon_0} \ar[r]^{\epsilon_1} & X
\enddiagram  \]
Note that the object $\Psi ^\square X$ is defined to be the pullback $\pullb{\Psi X}{\Psi (\Psi ^{\times 4} X)}{u}{\epsilon_0}$.

Now we let $\delta_X: \Psi X \to \Psi ^\square X$ be a solution to the following lifting problem.
\[ \diagram
X \ar[d]_{\eta } \ar[rr]^{\lambda_u \eta} & & \Psi ^\square X \ar[d]_{\rho_u}|=^{\epsilon_0 \times \epsilon_1 \times \zeta_0 \times \zeta_1} \\
\Psi X \ar[rr]_-{\eta\epsilon_0 \times 1 \times \eta\epsilon_0 \times 1} \ar@{-->}[urr]^{\delta_X}&& \Psi ^{\times 4} X \enddiagram \]

For any $f: X \to Y$, we need to find a solution to the following lifting problem in order to define $\tau_f: \pullb{X}{\Psi ^\square Y}{ \eta f }{\zeta} \to \Psi (\pullb{X}{\Psi Y}{f}{\epsilon_0}) $.

\[ \diagram
X \ar[d]_{1 \times   \lambda_u \eta f} \ar[rr]^{\eta (1 \times \eta f )} & & \Psi (\pullb{X}{\Psi Y}{f}{\epsilon_0}) \ar[d]^{\epsilon_0 \times \epsilon_1} \\
\pullb{X}{\Psi ^\square Y}{ \eta f }{\zeta_0} \ar[rr]_-{(1_X \times \epsilon_0) \times (1_X \times \epsilon_1)} \ar@{-->}[urr]^{\tau_f}&& (\pullb{X}{\Psi Y}{f}{\epsilon_0}) \times (\pullb{X}{\Psi Y}{f}{\epsilon_0}) \enddiagram \]

Since $R$ is an $\Id$-presentation of $(\coalg{\reltofact(R)}, \alg{\reltofact(R)})$, we know that the right hand map above, $\epsilon_0 \times \epsilon_1$, is in $\alg{\reltofact(R)}$. Thus, we need to show that $1 \times \lambda_u \eta f$ is in $\coalg{\reltofact(R)}$.

To see this, first observe that $\zeta_0 \times \zeta_1: \Psi ^{\times 4} Y \to \Psi Y \times \Psi Y$ is in $\alg{\reltofact(R)}$, since it is given by the following pullback.

\[ \diagram
\Psi ^{ \times 4} Y \ar[d]^{\zeta_0 \times \zeta_1}  \ar[rr]^{\epsilon_0 \times \epsilon_1 } \pullback & & \Psi Y \times \Psi Y \ar[d]^{\epsilon_0 \times \epsilon_1 \times \epsilon_0 \times \epsilon_1} \\
\Psi Y \times \Psi Y \ar[rr]^-{\epsilon_0 \times \epsilon_0 \times \epsilon_1 \times \epsilon_1} & & Y \times Y \times Y \times Y
\enddiagram \]
The right-hand map in the above diagram is in $\alg{\reltofact(R)}$ since it is the product of two maps in $\alg{\reltofact(R)}$, and thus its pullback, $\zeta_0 \times \zeta_1$, is also in $\alg{\reltofact(R)}$. Then the composition $(\zeta_0 \times \zeta_1)\rho_u: \Psi ^\square Y  \to  \Psi Y \times \Psi Y$, which we also denote by $\zeta_0 \times \zeta_1$, is in $\alg{\reltofact(R)}$.

Thus, the following is a factorization of the diagonal $\Delta_{\Psi Y}$ into $\coalg{\reltofact(R)}$ and $\alg{\reltofact(R)})$.
\[ \diagram 
\Psi Y \ar[r]^{\lambda_u} & \Psi ^\square Y \ar[r]^-{\zeta_0 \times \zeta_1}& \Psi Y \times \Psi Y
\enddiagram\]
By Corollary \ref{newfact}, in the following factorization of $\eta f: X \to \Psi Y$,
\[ X \xrightarrow{1 \times \lambda_u \eta f} \pullb{X}{\Psi ^\square Y}{ \eta f }{\zeta_0}  \xrightarrow{\zeta_1} \Psi Y\]
the morphism $1 \times \lambda_u \eta f$ is in $\coalg{\reltofact(R)}$.

Thus, we obtain a lift $\tau_f$ as above.

Then $\tau$ and $\delta$ make $R$ into a homotopical relation where the diagram 
\[ \diagram 
X \ar[r]^-\eta & \Psi ^\square X \ar@<1ex>[r]^{\epsilon_0} \ar[r]|{\epsilon_1} \ar@<-1ex>[r]_{\zeta}& \Psi  X
\enddiagram \]
of Definition \ref{homotopical} is given by the diagram
\[ \diagram 
X \ar[r]^-{\lambda_u \eta} & \Psi ^\square X \ar@<1ex>[r]^{\epsilon_0} \ar[r]|{\epsilon_1} \ar@<-1ex>[r]_{\zeta_0}& \Psi  X
\enddiagram \]
that we have defined here.
\end{proof}

Thus, we have the following theorem.

\begin{thm}\label{idpres2mrs}
Consider a relation $R$ on $\C$. It is an $\Id$-presentation of a weak factorization system if and only if it is a Moore relation system.
\end{thm}
\begin{proof}
By Proposition \ref{mrs2idpres}, a Moore relation system is an $\Id$-presentation of the weak factorization system it generates.

By Propositions \ref{idpres2trans}, \ref{idpres2sym}, and \ref{idpres2homo}, an $\Id$-presentation of a weak factorization system is a Moore relation system.
\end{proof}

Now we can restate Theorem \ref{wfs2idpres} in the following way.

\begin{cor}\label{ttwfs2mrs}
Consider a type-theoretic weak factorization structure $W$. Then the relation $\facttorel (W)$ is a Moore relation system which generates the weak factorization system represented by $W$. Thus, every type-theoretic weak factorization system can be generated by a Moore relation system.
\end{cor}
\begin{proof}
This is Theorem \ref{wfs2idpres} with `Moore relation system' substituted for `$\Id$-presentation' as justified by Theorem \ref{idpres2mrs}.
\end{proof}

\begin{exmp}
Consider Example \ref{cisex}. Then given a Cisinski model structure
$(\C, \F, \W)$ on a topos $\M$, the weak factorization system $(\C \cap \W \cap \M_{\F} , \F \cap \M_{\F})$ is generated by a Moore relation system. In particular, the weak factorization systems in the category of Kan complexes, the category of quasicategories, and that of cubical sets are generated by Moore relation systems.
\end{exmp}



To conclude this section, we show that $|  \facttorel \reltofact  | :| \IdPres[\C] |\to| \IdPres[\C] |$ is isomorphic to the identity functor. We have shown that $| \reltofact \facttorel  |:| \ttWFS[\C] |\to| \ttWFS[\C] |$ is also isomorphic to the identity functor (Corollary \ref{wfs2idpresfunctors}). Thus, this will show that $|   \facttorel |$ and $|  \reltofact |$ form an equivalence
$| \IdPres[\C] | \simeq | \ttWFS[\C] |$.

\begin{prop}\label{idpresiso}
The functor $|  \facttorel \reltofact  | :| \IdPres[\C] |\to| \IdPres[\C] |$ is isomorphic to the identity functor.
\end{prop}
\begin{proof} 
We need to provide an equivalence between any $ R$ in $\IdPres[\C]$ and $ \facttorel \reltofact (R)$. Since $| \IdPres[\C] |$ is a proset, this will automatically assemble into a natural isomorphism $1 \cong |  \facttorel \reltofact  | $.

Let $ R X$ be denoted by the following diagram for any $X$ in $\C$.
\[ \diagram
 X   \ar[r]|-{\eta_X} &   \Psi  X  \ar@<1.5ex>[l]^-{\epsilon_{1X} } \ar@<-1.5ex>[l]_-{\epsilon_{0X} }
\enddiagram \]
Then $ \facttorel \reltofact (R X) $ gives the following diagram
\[ \diagram
 X   \ar[rr]|-{\lambda(\Delta_X)} & & \pullb{X}{ \Psi ( X \times X) }{\Delta}{\epsilon_0} \ar@<1.5ex>[ll]^-{\pi_1 \rho(\Delta_X) } \ar@<-1.5ex>[ll]_-{\pi_0 \rho(\Delta_X)}
\enddiagram \]
where $\lambda$ denotes $\lambda_{\reltofact ( R)}$ and $\rho$ denotes $\rho_{\reltofact (  R)}$.

Now a morphism $ R \to \facttorel \reltofact (R)$ consists of a natural transformation $ R(X) \to \facttorel \reltofact (R)(X)$ at each $X$, as displayed below, which, in turn, consists of the identity on $X$ and a morphism $\tau_X: \Psi X \to \pullb{X}{ \Psi ( X \times X) }{\Delta}{\epsilon_0}$.
\[ \diagram
 X   \ar[rr]|-{\eta_X} \ar@{=}[dd]&&   \Psi  X  \ar@<1.5ex>[ll]^-{\epsilon_{1X} } \ar@<-1.5ex>[ll]_-{\epsilon_{0X} } \ar@{-->}[dd]^{\tau_X} \\ \\
  X   \ar[rr]|-{\lambda(\Delta_X) } & & \pullb{X}{ \Psi ( X \times X) }{\Delta}{\epsilon_0} \ar@<1.5ex>[ll]^-{\pi_1 \rho(\Delta_X)  } \ar@<-1.5ex>[ll]_-{\pi_0 \rho(\Delta_X) }
\enddiagram \]
But we can obtain the morphism $\tau_X$ as a lift in the diagram below.
\[ \diagram
X \ar[r]^-{\lambda (\Delta_X)} \ar[d]_\eta & \pullb{X}{ \Psi ( X \times X) }{\Delta}{\epsilon_0} \ar[d]^{\rho(\Delta_X)} \\
\Psi X \ar[r]_{\epsilon_0 \times \epsilon_1} \ar@{-->}[ur]^{\tau_X} & X \times X
\enddiagram\]
since $\eta$ is in $\coalg{\reltofact(R)}$ and $\rho(\Delta_X)$ is in $\alg{\reltofact(R)}$.

Similarly, we can get a morphism $ \facttorel \reltofact (R) \to  R$ by solving the following lifting problem for each object $X$.
\[ \diagram
X \ar[d]_{\lambda (\Delta_X)} \ar[r]^-\eta & \Psi X \ar[d]^{\epsilon_0 \times \epsilon_1} \\
\pullb{X}{ \Psi ( X \times X) }{\Delta}{\epsilon_0} \ar[r]_-{\rho(\Delta_X)} \ar@{-->}[ur] & X \times X
\enddiagram\]
These lifts exist since $\lambda(\Delta_X)$ is in $\coalg{\reltofact(R)}$ and $\epsilon_0 \times \epsilon_1$ is in $\alg{\reltofact(R)}$.
\end{proof}

\section{Models of $\Id$-types}\label{sec:idtypes}
In this section, we show that a weak factorization system on a finitely complete category $\C$ models $\Id$-types if and only if it models $\Id$-types on objects. We use this result to also show that if a display map category $(\C,\D)$ models $\Id$-types, then $(\C,(^\boxslash \D)^\boxslash)$ also models $\Id$-types regardless if they are functorial (this is not used in our main result, but was mentioned in Section \ref{sec:dmc} to motivate our work).

Consider $\Id$-presentation $R$ which at each object $X$ of $\C$ has components denoted as follows.
\[ \diagram
 X   \ar[r]|-{\eta_X} &  \Psi X , \ar@<1.5ex>[l]^-{\epsilon_{1X} } \ar@<-1.5ex>[l]_-{\epsilon_{0X} }
\enddiagram \]
For any object $Y$ of $\C$, let $\A/ Y$ denote the full subcategory of $\C / Y$ spanned by $\alg{\reltofact(R)}$. Let $\dom_Y: \A / Y \to \C$ denote the domain functor. The weak factorization structure $\reltofact (R)$ induces a weak factorization structure $\reltofact (R)_Y$ \cite[Thm.~15.3.6]{MP12} which takes a morphism $\alpha : f \to g$ of $\A/Y$ to the following.
\[ \diagram
W \ar[drr]_f \ar[rr]^-{1_W \times  \eta \dom_Y (\alpha)} && W \times_X \Id(X) \ar[d]|(.45){g\epsilon_1 \pi_{\Id(X)}} \ar[rr]^-{\epsilon_1 \pi_{\Id(X)}} && X\ar[dll]^-{g } \\
&& Y
\enddiagram \]
The classes of coalgebras and algebras of this weak factorization structure in $\A/Y$ are exactly the preimages of those in $\C$: $\coalg{\reltofact (R)_Y}$ = $\dom_Y^{-1} (\coalg{\reltofact (R)})$ and $\alg{\reltofact (R)_Y}= \dom_Y^{-1}(\alg{\reltofact (R)})$.

\begin{lem}\label{Idtypesonobj2morlem}
The weak factorization structure $\reltofact (R)_Y$ on $\A/Y$ is type-theoretic.
\end{lem}
\begin{proof}
First, note that $1_Y \in \alg{\reltofact (R)}$, so $1_Y$ is a terminal object of $\A/Y$. For any $f \in \A/Y$, the morphism to this terminal object is $f: f \to 1_Y$, whose image under $\dom_Y$ is $f \in \alg{\reltofact (R)}$. Thus, $f$ is fibrant.

Now consider any $\ell \in \coalg{\reltofact (R)_Y}$ and $r \in \alg{\reltofact (R)_Y}$ who share a codomain. We have that $\dom_Y(r^* \ell) = \dom_Y(r)^* \dom_Y(\ell) \in \coalg{\reltofact (R)}$ since $\reltofact (R)$ is type-theoretic (Theorem \ref{idpres2ttwfs}). Therefore, $r^* \ell \in \coalg{\reltofact (R)_Y}$.
\end{proof}

\begin{prop}\label{Idtypesonobj2mor}
Consider a weak factorization system with an $\Id$-presentation $R$. Then $(\C, \alg{\reltofact(R)})$ models $\Id$-types.
\end{prop}

\begin{proof}
Since the weak factorization structure $\reltofact (R)_Y$ on $\C/Y$ is type-theoretic, it
has an $\Id$-presentation $  \facttorel (\reltofact (R)_Y)$ which at each $f: X \to Y$ gives the following relation (depicted as a diagram in $\C$).
\[ \diagram
 X   \ar[rr]|-{1_X \times \eta \Delta} \ar@/_/[dr]_{f} & & X {_{\Delta}\times_{\epsilon_0}} \Id(X \times_Y X)   \ar@<1.5ex>[ll]^-{\pi_1 \epsilon_1 \pi_1} \ar@<-1.5ex>[ll]_-{\pi_0 \epsilon_1 \pi_1} \ar@/^/[dl] \\
 & Y
\enddiagram \tag{$*$}\]

Now we show that the collection of these $\Id$-presentations is a model of $\Id$-types in the display map category $(\C, \alg{\reltofact(R)})$. We just need to check that any pullback $\alpha^* (1_X \times \eta \Delta)$ as shown below
of $1_X \times \eta \Delta$ along any morphism 
$\alpha: A \to X$ for $i = 0,1$ is in 
$\coalg{\reltofact (R)}$. 
\[ \diagram 
& X {_{\Delta}\times_{\epsilon_0}} \Id(X \times_Y X) {_{\pi_i \epsilon_1} \times_{\alpha}} A \ar[rr] \ar[dd]|!{"2,1";"2,4"}\hole &&{X {_{\Delta}\times_{\epsilon_0}} \Id(X \times_Y X)} \ar[dd]^{\pi_i \epsilon_1 \pi_1} \\
A \ar[rr] \ar@{=}[dr] \ar[ur]^-{\alpha^* (1_X \times \eta \Delta)} && X \ar@{=}[dr] \ar[ur]^-{1_X \times \eta \Delta} & \\
& A \ar[rr]^\alpha & & X
\enddiagram \]

Note that the image of diagram $(*)$ above is a relation of $X$ in $\C$ where $1_X \times \eta \Delta \in \coalg{\reltofact( R)}$ and $\epsilon_1 \pi_1 = \pi_0 \epsilon_1 \pi_1 \times \pi_1  \epsilon_1 \pi_1 \in \alg{\reltofact( R)}$. Then by Corollary \ref{newfact}, the pullbacks shown above are in $\coalg{\reltofact( R)}$ when $i=0$. Similarly, by considering the involution $I$ which swaps $\epsilon_0$ and $\epsilon_1$ (described in the proof of \ref{rel2idpres}) applied to this relation, the pullbacks shown above are in $\coalg{\reltofact( R)}$ when $i=1$ by Corollary \ref{newfact}.
\end{proof}

\begin{cor}\label{dropfunc}
Consider a display map category $(\C,\D)$ which models $\Sigma$- and $\Id$-types. Then $(\C,(^\boxslash \D)^\boxslash)$ also models $\Id$-types.
\end{cor}
\begin{proof}
Let $\{I_Y\}_{Y \in \C}$ be the model of $\Id$-types in $(\C,\D)$. We obtain a model $I_*$ of $\Id$-types on objects in $(\C, \D)$ (for a terminal object $*$ in $\C$). But $I_*$ is trivially also a model of $\Id$-types on objects in $(\C, (^\boxslash \D)^\boxslash)$. By Proposition \ref{Idtypesonobj2mor}, $(\C, (^\boxslash \D)^\boxslash)$ also models $\Id$-types.
\end{proof}

\section{The main result}
\label{sec:summary}

In Section \ref{sec:relandfact}, we defined the following diagram of categories.
\[
 \diagram
 \fIdPres[\C] \ar@{^{ (}->}[r] \ar@{^{ (}->}[d]  & \frel[\C]  \ar@{^{ (}->}[d]  \ar@<0.5ex>[r]^-{\reltofact} &  \ffact[\C] \ar@{^{ (}->}[d] \ar@<0.5ex>[l]^-{\facttorel} & \fttWFS[\C] \ar@{_{ (}->}[l]  \ar@{^{ (}->}[d]  \\
 \IdPres[\C]  \ar@{^{ (}->}[r] &   \rel[\C] \ar@<0.5ex>[r]^-{\reltofact}  & \fact[\C]  \ar@<0.5ex>[l]^-{\facttorel}&   \ttWFS[\C] \ar@{_{ (}->}[l]
\enddiagram 
\]

In Section \ref{sec:mrs}, we defined (strict) Moore relation systems. We showed that strict Moore relation systems generate type-theoretic, algebraic weak factorization systems, and that Moore relation systems generate type-theoretic weak factorization systems.

In Sections \ref{sec:bigdiagram} and \ref{sec:idpresandmrs}, we showed the following theorem.

\begin{thm}\label{mainthm}
Consider a category $\C$ with finite limits.
The functors $  \facttorel  $ and $  \reltofact $ described above restrict to functors shown below.
\[
 \diagram
 \fIdPres[\C]  \ar@{^{ (}->}[d]  \ar@<0.5ex>[r]^-{\reltofact} & \ar@{^{ (}->}[d] \ar@<0.5ex>[l]^-{\facttorel}  \fttWFS[\C]  \\
 \IdPres[\C]   \ar@<0.5ex>[r]^-{\reltofact} & \ar@<0.5ex>[l]^-{\facttorel}  \ttWFS[\C]  
\enddiagram 
\]
Furthermore, after application of the proset truncation, these are equivalences.
\[
 \diagram
| \fIdPres[\C] | \ar@{^{ (}->}[d]  \ar@<0.75ex>[r]^-{|\reltofact|} \ar@{}[r]|-\sim & \ar@{^{ (}->}[d] \ar@<0.75ex>[l]^-{|\facttorel|}  {| \fttWFS[\C] | } \\
 |\IdPres[\C]|   \ar@<0.75ex>[r]^-{|\reltofact|} \ar@{}[r]|-\sim & \ar@<0.75ex>[l]^-{|\facttorel|} {| \ttWFS[\C] | }
 \enddiagram \tag{$*$}
\]
\end{thm}

\begin{proof}
The fact that $ \facttorel $ and $\reltofact $ restrict to functors $ \ttWFS[\C] \leftrightarrows \IdPres[\C] $ is proven in Corollary \ref{wfs2idpresfunctors} and Corollary \ref{idpres2ttwfs}. 
Then consider an object or morphism $X$ of $ \fttWFS[\C]$. We know that $ \facttorel (X) \in  \IdPres[\C] \cap \frel[\C] = \fIdPres[\C]$ so $ \facttorel $ restricts to a functor 
$ \fttWFS[\C] \to \fIdPres[\C] $. Similarly, $\reltofact $ restricts to a functor $\fIdPres[\C] \to\fttWFS[\C]$.

The fact that $| \facttorel |$ and $|\reltofact |$ constitute an equivalence $|\IdPres[\C] | \simeq | \ttWFS[\C]|$ is proven in Corollary \ref{wfs2idpresfunctors} and Proposition \ref{idpresiso}. Since both squares in the diagram $(*)$ above commute, this restricts to an equivalence $ |\fIdPres[\C] | \simeq | \fttWFS[\C]|$.
\end{proof}

%

We then interpret this in the following theorem.
\begin{thm}\label{mainthmcor}
Consider a category $\C$ with finite limits.
The following properties of any weak factorization system $(\mathsf L, \mathsf R)$ on $\C$ are equivalent:
\begin{enumerate}
\item it has an $\Id$-presentation;
\item it is type-theoretic;
\item it is generated by a Moore relation system;
\item $(\C, \mathsf R)$ is a display map category modeling $\Sigma$- and $\Id$-types.
\end{enumerate}
\end{thm}
\begin{proof}
The equivalence between (1) and (3) appears as Theorem \ref{idpres2mrs}.

That (2) implies (1) is Theorem \ref{wfs2idpres}.

That (3) implies (2) is Theorem \ref{mrs2ttwfs}.

Clearly, (4) implies (1).

By Theorem \ref{relfact2fib}, (1) implies that all objects are fibrant in $(\mathsf L, \mathsf R)$, and thus $(\C, \mathsf R)$ is a display map category modeling $\Sigma$-types (Example \ref{wfssigma}). By Proposition \ref{Idtypesonobj2mor}, (1) implies that $(\C, \mathsf R)$ models $\Id$-types. Thus, (1) implies (4).
\end{proof}

\begin{thm}\label{mainthmcor2}
Consider a category $\C$ with finite limits and a weak factorization system $(\mathsf L, \mathsf R)$ satisfying the equivalent statements of the preceding theorem, \ref{mainthmcor}. 

If $(\C, \mathsf R)$ models pre-$\Pi$-types, then it models $\Pi$ types.
In particular, if $\C$ is locally cartesian closed, then $(\C, \mathsf R)$ models $\Pi$-types.
\end{thm}
\begin{proof}
By the previous theorem, \ref{mainthmcor}, $(\mathsf L, \mathsf R)$ is type-theoretic. Therefore, by Proposition \ref{pre2pitypes}, $(\C, \mathsf R)$ models $\Pi$-types.
\end{proof} 

\addcontentsline{toc}{section}{References}
\bibliographystyle{halpha}
\bibliography{article}

\end{document}